\documentclass[reqno]{amsart}

\usepackage{amsfonts}
\usepackage[all]{xy}
\usepackage{amssymb}
\usepackage{amsmath}
\usepackage{epsfig}
\usepackage{amscd}
\usepackage{graphicx, color}
\usepackage{epstopdf}

\def\E{\ifmmode{\mathbb E}\else{$\mathbb E$}\fi} 
\def\N{\ifmmode{\mathbb N}\else{$\mathbb N$}\fi} 
\def\R{\ifmmode{\mathbb R}\else{$\mathbb R$}\fi} 
\def\Q{\ifmmode{\mathbb Q}\else{$\mathbb Q$}\fi} 
\def\C{\ifmmode{\mathbb C}\else{$\mathbb C$}\fi} 
\def\H{\ifmmode{\mathbb H}\else{$\mathbb H$}\fi} 
\def\Z{\ifmmode{\mathbb Z}\else{$\mathbb Z$}\fi} 
\def\P{\ifmmode{\mathbb P}\else{$\mathbb P$}\fi} 
\def\T{\ifmmode{\mathbb T}\else{$\mathbb T$}\fi} 
\def\SS{\ifmmode{\mathbb S}\else{$\mathbb S$}\fi} 
\def\DD{\ifmmode{\mathbb D}\else{$\mathbb D$}\fi} 

\newcommand{\e}{\varepsilon}

\newcommand{\del}{\partial}

\newcommand{\ben}{\begin{enumerate}}
\newcommand{\een}{\end{enumerate}}
\newcommand{\be}{\begin{equation}}
\newcommand{\ee}{\end{equation}}
\newcommand{\bea}{\begin{eqnarray}}
\newcommand{\eea}{\end{eqnarray}}
\newcommand{\beastar}{\begin{eqnarray*}}
\newcommand{\eeastar}{\end{eqnarray*}}
\newcommand{\bc}{\begin{center}}
\newcommand{\ec}{\end{center}}

\theoremstyle{theorem}
\newtheorem{thm}{Theorem}[section]
\newtheorem{cor}[thm]{Corollary}
\newtheorem{lem}[thm]{Lemma}
\newtheorem{prop}[thm]{Proposition}

\theoremstyle{definition}
\newtheorem{defn}{Definition}[section]
\newtheorem{rem}[defn]{Remark}

\newtheorem{ques}[defn]{Question}

\newtheorem{exm}[defn]{Example}

\newtheorem*{thm*}{Theorem}

\numberwithin{equation}{section}

\def\R{{\mathbb R}}
\def\osc{{\hbox{\rm osc}}}
\def\Crit{{\hbox{Crit}}}

\def\E{{\mathbb E}}
\def\Z{{\mathbb Z}}
\def\C{{\mathbb C}}
\def\R{{\mathbb R}}
\def\P{{\mathbb P}}

\def\N{{\mathbb N}}

\def\11{{\mathbb I}}

\def\H{\mathbb{H}}

\def\delbar{{\overline \partial}}

\def\C{\mathbb{C}}
\def\Z{\mathbb{Z}}

\def\T{\mathbb{T}}

\def\Q{\mathbb{Q}}

\def\E{\ifmmode{\mathbb E}\else{$\mathbb E$}\fi} 
\def\N{\ifmmode{\mathbb N}\else{$\mathbb N$}\fi} 
\def\R{\ifmmode{\mathbb R}\else{$\mathbb R$}\fi} 
\def\Q{\ifmmode{\mathbb Q}\else{$\mathbb Q$}\fi} 
\def\C{\ifmmode{\mathbb C}\else{$\mathbb C$}\fi} 
\def\Z{\ifmmode{\mathbb Z}\else{$\mathbb Z$}\fi} 
\def\P{\ifmmode{\mathbb P}\else{$\mathbb P$}\fi} 
\def\CS{\ifmmode{\mathbb S}\else{$\mathbb S$}\fi} 
\def\DD{\ifmmode{\mathbb D}\else{$\mathbb D$}\fi} 

\def\R{{\mathbb R}}
\def\osc{{\hbox{\rm osc}}}
\def\Crit{{\hbox{Crit}}}
\def\E{{\mathbb E}}
\def\Z{{\mathbb Z}}
\def\C{{\mathbb C}}
\def\R{{\mathbb R}}

\def\N{{\mathbb N}}

\def\delbar{{\overline \partial}}

\def\e{\varepsilon}

\def\CA{{\mathcal A}}

\def\CC{{\mathcal C}}

\def\CM{{\mathcal M}}

\def\CP{{\mathcal P}}

\def\CP{{\mathcal P}}
\def\CS{{\mathcal S}}
\def\CT{{\mathcal T}}

\def\darr#1{\raise1.5ex\hbox{$\leftrightarrow$}
\mkern-16.5mu #1}

\def\roughly#1{\raise.3ex\hbox{$#1$\kern-.75em
\lower1ex\hbox{$\sim$}}}

\def\opname#1{\mathop{\kern0pt{\rm #1}}\nolimits}

\def\Im{\opname{Im}}

\def\dim{\opname{dim}}

\def\dudtau{\frac{\partial u}{\partial \tau}}
\def\dudt{\frac{\partial u}{\partial t}}
\def\dvdtau{\frac{\partial v}{\partial \tau}}
\def\dvdt{\frac{\partial v}{\partial t}}
\def\supp{\operatorname{supp}}

\def\Graph{\operatorname{Graph}}

\def\Spec{\operatorname{Spec}}

\def\Crit{\operatorname{Crit}}

\def\leng{\operatorname{leng}}

\def\Int{\operatorname{Int}}
\def\Sing{\operatorname{Sing}}

\begin{document}
\quad \vskip1.375truein

\def\mq{\mathfrak{q}}
\def\mH{\mathfrak{H}}
\def\mh{\mathfrak{h}}
\def\ma{\mathfrak{a}}
\def\ms{\mathfrak{s}}
\def\mm{\mathfrak{m}}
\def\mn{\mathfrak{n}}

\def\Hoch{{\tt Hoch}}
\def\mt{\mathfrak{t}}
\def\ml{\mathfrak{l}}
\def\mT{\mathfrak{T}}
\def\mL{\mathfrak{L}}
\def\mg{\mathfrak{g}}
\def\md{\mathfrak{d}}

\title[generating functions and spectral invariants]
{Geometry of generating functions and Lagrangian spectral invariants}
\author{Yong-Geun Oh}
\thanks{This work is supported by the Institute for Basic Sciences}

\address{Center for Geometry and Physics, Institute for Basic Sciences (IBS), Pohang, Korea \&
Department of Mathematics, POSTECH, Pohang, KOREA,\&
Department of Mathematics, University of Wisconsin, Madison, WI
53706}
\email{yongoh@ibs.re.kr,oh@math.wisc.edu}

\begin{abstract} Partially motivated by
the study of topological Hamiltonian dynamics, we prove various $C^0$-aspects
of the Lagrangian spectral invariants and the basic phase functions $f_H$, that is,
a natural graph selector constructed by Lagrangian Floer homology of $H$
(relative to the zero section $o_N$). In particular, we prove that
$$
\gamma^{lag}(\phi_H^1(o_N)): = \rho^{lag}(H;1) -  \rho^{lag}(H;[pt]^\#) \to 0
$$
as $\phi_H^1 \to id$, \emph{provided} $H$'s satisfy
$\supp X_H \subset D^R(T^*N) \setminus o_B$ for some $R > 0$ and
a closed subset $B \subset N$ with nonempty interior.

We also study the relationship between $f_H$ and $\rho^{lag}(H;1)$ and
prove a structure theorem of the micro-support of the singular locus $\Sing(\sigma_H)$ of
the function $f_H$. Based on this structure theorem and a classification
theorem of generic Lagrangian singularity in $\dim N = 2$ obtained by Arnold's school,
we define the notion of cliff-wall surgery when $\dim N = 2$:
the surgery replaces a multi-valued Lagrangian graph $\phi_H^1(o_N)$
by a piecewise-smooth Lagrangian cycle that is
canonically constructed out of the single valued branch
$\Sigma_H: = \Graph df_H \subset \phi_H^1(o_N)$ defined on an open dense subset of
$N \setminus \Sing(\sigma_H)$ of codimension 1.
\end{abstract}


\date{June 13, 2012; revised, June 4, 2013}

\maketitle

\medskip

\tableofcontents

\section{Introduction}
\label{sec:vanishing}

We always assume that the ambient manifold $M$  or $N$ are connected throughout the entire paper.

\subsection{Weak Hamiltonian topology of $Ham(M,\omega)$}
\label{subsec:top-flows}

In \cite{oh:hameo1}, M\"uller and the author introduced the notion of
Hamiltonian topology on the subset of the space $\CP(Homeo(M),id)$ of continuous
paths on $Homeo(M)$ consisting of Hamiltonian paths $\lambda:[0,1] \to Symp(M,\omega)$ with
$\lambda(t) = \phi_H^t$ for some time-dependent Hamiltonian $H$. We denote this subset by
$$
\CP^{ham}(Symp(M,\omega),id).
$$
We would like to emphasize that we do \emph{not} assume that $H$
is normalized unless otherwise said explicitly. This is because we need to
consider both compactly supported and mean-normalized Hamiltonians
and suitably transform one to the other in the course of the proof of
the various theorems of this paper.

In this subsection, we first recall the definition from \cite{oh:hameo1}
of the Hamiltonian topology mostly restricted on the open manifold $T^*N$.
While \cite{oh:hameo1} considers \emph{strong} Hamiltonian topology, except
Remark 3.27 therein, the more relevant topology in the present paper will be
the \emph{weak} Hamiltonian topology. We first recall its definition.

For a given continuous function $h: M \to \R$, we denote
$$
\osc(h) = \max h - \min h.
$$
We define the $C^0$-distance $\overline d$ on $Homeo(M)$ by the symmetrized
$C^0$-distance
$$
\overline d(\phi,\psi) = d_{C^0}(\phi,\psi) + d_{C^0}(\phi^{-1},\psi^{-1})
$$
and the $C^0$-distance on $\CP(Homeo(M),id)$, again denoted by $\overline d$,
by
$$
\overline d(\lambda,\mu) = \max_{t \in [0,1]} \overline d(\lambda(t),\mu(t)).
$$
This induces the corresponding $C^0$-distance on $\CP^{ham}(Symp(M,\omega),id)$.
The Hofer length of Hamiltonian path $\lambda = \phi_H$ is defined by
$$
\leng(\lambda) = \int_0^1 \osc(H_t)\, dt = \|H\|.
$$
Following the notations of \cite{oh:hameo1}, we denote by $\phi_H$ the
Hamiltonian path
$$
\phi_H: t \mapsto \phi_H^t; \, [0,1] \to Ham(M,\omega)
$$

\begin{defn}\label{defn:hamtopology} Let $(M,\omega)$ be an open symplectic
manifold. Let $\lambda, \, \mu$ be smooth Hamiltonian paths with compact support in
$\Int M$. The weak Hamiltonian topology is the metric topology induced by the
metric
\be\label{eq:strong} d_{ham}(\lambda,\mu): = \overline
d(\lambda(1),\mu(1)) + \operatorname{leng}(\lambda^{-1}\mu).
\ee
\end{defn}

\subsection{Hamiltonian $C^0$-topology on $\frak{Iso}_B(o_N;T^*N)$}
\label{subsec:top-flows}

Let $N$ be a closed smooth manifold.
We equip the cotangent bundle $T^*N$ with  the Liouville one-form $\theta$ defined by
$$
\theta_x(\xi_x) = p(d\pi(\xi_x)), \quad x = (q,p) \in T^*N.
$$
The canonical symplectic form $\omega_0$ on $T^*N$ is defined by
\be\label{eq:omega0} \omega_0 = -d\theta = \sum_{k=1}^n dq^k \wedge
dp_k \ee where $(q^1,\ldots, q^n,p_1,\ldots, p_n)$ is the canonical
coordinates of $T^*N$ associated to the coordinates $(q^1,\ldots,
q^n)$ of $N$.

Consider Hamiltonian $H = H(t,x)$ such that $H_t$ is asymptotically constant, i.e.,
the one whose Hamiltonian vector field $X_H$ is compactly supported. We define
$$
\supp_{asc}H = \supp X_H : = \bigcup_{t \in [0,1]} X_{H_t}.
$$
For each given $K,\, R \in \R_+$, we define
\be\label{eq:CKR}
\CP C_{R,K}^\infty = \{H \in C^\infty([0,1]\times T^*N,\R) \mid \supp_{asc}H \subset D^R(T^*N),
\|H\| \leq K\}
\ee
which provides a natural filtration of the space $C^\infty([0,1]\times T^*N,\R)$.
We also denote
\bea\label{eq:CR}
\CP C_{R}^\infty = \bigcup_{K \in \R_+} \CP C_{R,K}^\infty, \quad
\CP C_{asc}^\infty  =  \bigcup_{R \geq 0} \CP C_R^\infty.
\eea
By definition, each element $H_t$ is independent of $x = (q,p)$ if $|p|$
is sufficiently large and so carries a smooth function $c_\infty: [0,1] \to \R$ defined by
$$
c_\infty(t) = H(t,\infty).
$$
Therefore we have the natural evaluation map
$$
\pi_\infty:\CP C_{asc}^\infty \to C^\infty([0,1],\R).
$$
For each given smooth function $c:[0,1] \to \R$,
we denote
\be\label{eq:PPascc}
\CP C^\infty_{asc;c}: = \pi_\infty^{-1}(c).
\ee
We then introduce the space of Hamiltonian deformations of
the zero section and denote
$$
\mathfrak{Iso}(o_N;T^*N) = \{\phi_H^1(o_N) \mid H \in \CP C_{asc}^\infty\}
$$
following the terminology of \cite{alan:isodrast}, and
\bea\label{eq:IsoK}
\mathfrak{Iso}(o_N;D^R(T^*N)) & = & \{\phi_H^1(o_N) \mid H \in \CP C_{R}^\infty\}\nonumber\\
\mathfrak{Iso}^K(o_N;D^R(T^*N)) & = & \{\phi_H^1(o_N) \mid H \in \CP C_{R,K}^\infty\}.
\eea

Now we equip a topology with $\mathfrak{Iso}(o_N;T^*N)$. One needs to pay some attention
in finding the correct definition of the topology suitable for the study of
Hamiltonian geometry of the set $\mathfrak{Iso}(o_N;T^*N)$. For this purpose, we
introduce the following measurement of $C^0$-fluctuation of the Hamiltonian diffeomorphism of $\phi_F^1$
along the zero section $o_N \subset T^*N$,
$$
\osc_{C^0}(\phi_F^1;o_N): = \max\left\{\max_{x \in o_N}d(\phi_F^1(x),x),
\max_{x \in o_N} d((\phi_F^1)^{-1}(x),x)\right\}.
$$
Using this we introduce the following restricted $C^0$-distance

\begin{defn}\label{defn:hamC0dist}
Let $L_0, \, L_1 \in \mathfrak{Iso}(o_N;T^*N)$
with $L_0 = \phi_{F^0}^1(o_N), \, L_1=\phi_{F^1}^1(o_N)$. We
define the following distance function
\bea\label{eq:dhamC0}
d^{ham}_{C^0}(L_0,L_1) & = & \inf_{\{H;\phi_H^1(L_0) = L_1\}} \max
\left\{\osc_{C^0}\left((\phi_{F^1}^1)^{-1}\phi_H^1\phi_{F^0}^1(o_N);o_N\right)\right., \nonumber\\
&{}& \hskip1in \left.\osc_{C^0}\left((\phi_{F^0}^1)^{-1}(\phi_H^1)^{-1}\phi_{F^1}^1(o_N);o_N\right)\right\}
\eea
on $\mathfrak{Iso}^K(o_N;D^R(T^*N))$, which induces the metric topology thereon. We equip with
$\mathfrak{Iso}(o_N;T^*N)$ the direct limit topology of $\mathfrak{Iso}^K(o_N;D^R(T^*N))$ as $R,\, K \to \infty$
and call it the \emph{Hamiltonian $C^0$-topology} of $\mathfrak{Iso}(o_N;T^*N)$.
\end{defn}

For the main theorems proved in the present paper, we will also need to consider the following
subset of Hamiltonian functions $H$.

Let $B \subset N$ be a given closed subset and $o_B \subset o_N$
the corresponding subset of the zero section. Denote by $T$ an open neighborhood of $o_B$ in
$T^*N$. We define
\be\label{eq:CascB}
\CP C_{asc;B}^\infty = \{H \in C^\infty([0,1]\times T^*N,\R) \mid \supp X_H \subset (T^*N \setminus B)
\, \text{is compact} \}.
\ee
We have the filtration
$$
\CP C_{asc;B}^\infty  = \bigcup_{T \supset B}  \bigcup_{R > 0} \CP C_{R;T}^\infty
$$
over the set of open neighborhoods $T$ of $B$ and the positive numbers $R> 0$
where
\be\label{eq:CRKT}
\CP C_T^\infty = \{H \in \CP C_{asc;B}^\infty \mid \phi_H^1 \equiv id \, \text{on }\,  T\}.
\ee
Here we would like to emphasize that \emph{the support condition on $T \supset o_B$
is imposed only for the time-one map $\phi_H^1$, but not for the whole path $\phi_H$}.
This indicates relevance of the following discussion to the weak Hamiltonian topology
described above.

Similarly as $\CP C_{R,K}^\infty$ above, we define $\CP C_{R,K;T}^\infty$.
We define $\frak{Iso}_B(o_N;T^*N)$ to be the subset
$$
\frak{Iso}_B(o_N;T^*N) = \{\phi_H^1(o_N) \mid H \in \CP C^\infty_{asc;B}\}.
$$
This has the filtration
$$
\frak{Iso}_B(o_N;T^*N) = \bigcup_{K\geq 0} \bigcup_{T \supset B} \frak{Iso}_T^K(o_N;T^*N;T)
$$
where
$$
\frak{Iso}^K_T(o_N;D^R(T^*N)) = \{\phi_H^1(o_N) \mid H \in \CP C_{R,K;T}^\infty\}.
$$
\begin{defn}\label{defn:hamC0onB} Equip with $\frak{Iso}_T^K(o_N;T^*N)$ the subspace topology of
the Hamiltonian $C^0$-topology of $\frak{Iso}(o_N;T^*N)$.
We then put on $\frak{Iso}_B(o_N;T^*N)$  the direct limit topology of $\frak{Iso}_T^K(o_N;T^*N)$
over $T \supset o_B$ and $K\geq 0$. We call this topology the Hamiltonian $C^0$-topology
of $\frak{Iso}_B(o_N;T^*N)$.
\end{defn}

Unravelling the definition, we can rephrase the meaning of the convergence $L_i \to L$
in $\frak{Iso}_B(o_N;T^*N)$ into the existence of $R, \, K > 0$, $T \supset o_B$ and a sequence
$H_i$ such that $L_i = \phi_{H_i}^1(L)$ and
\begin{enumerate}
\item $\|H_i\|\leq K$ for all $i$,
\item $\supp X_{H_i} \subset D^R(T^*N) \setminus o_B$ for all $i$,
\item $\phi_{H_i}^1 \equiv id$ on $T$ for all $i$,
\item $d^{ham}_{C^0}(L_i,L) \to 0$ as $i \to \infty$.
\end{enumerate}

\begin{rem}
\begin{enumerate}
\item We refer to the proof of Lemma \ref{lem:(T,f)} and Remark \ref{rem:support} for the reason
to take these particular support hypotheses (2), (3) imposed in our definition
of Hamiltonian $C^0$-topology of $\frak{Iso}_B(o_N;T^*N)$. This topology may be regarded as the Lagrangian analog to
the above mentioned weak Hamiltonian topology and seems to be the weakest possible
topology with respect to which one can prove the $C^0$-continuity of spectral copacity
$\gamma^{lag}$ which is stated in Theorem \ref{thm:continuity} below.
\item
The Lagrangianization $\Graph \phi_F^1$ of  Hamiltonian $F:[0,1] \times M \to \R$
with $\supp F \subset M \setminus B$, there exists an open neighborhood $U \supset B$
such that $\supp F \supset M \setminus U$. Therefore provided the $C^0$-distance of
$\overline d(\phi_F,id)=:\epsilon$ is so small that its
graph is contained in a Weinstein neighborhood of the diagonal, such a graph will automatically satisfy
$$
\phi_{\mathbb F}^1(o_{\Delta_B}) \subset T_\epsilon; \quad \mathbb F(t,{\bf x}) : = F(t,x), \, {\bf x} = (x,y)
$$
where $T_\epsilon$ is the $\epsilon$-neighborhood of $o_{\Delta_B}$ in $T^*\Delta$ and hence is
automatically contained in $\frak{Iso}_{o_{\Delta_B}}(o_\Delta,T^*\Delta)$.
\end{enumerate}
\end{rem}

\subsection{Lagrangian spectral invariants}
\label{subsec:lag-spectral}

For any given time-dependent Hamiltonian
$H = H(t,x)$, the classical action functional on the space
$$
\CP(T^*N) : = C^\infty([0,1],T^*N)
$$
is defined by
$$
\CA^{cl}_H(\gamma) = \int \gamma^*\theta - \int_0^1 H(t,\gamma(t))\, dt.
$$
We define the subset $\CP(T^*N;o_N)$ by
$$
\CP(T^*N;o_N) = \{\gamma:[0,1] \to T^*N \mid \gamma(0) \in o_N \}.
$$
The assignment $\gamma \mapsto \pi(\gamma(1))$ defines a fibration
$$
\CP(T^*N;o_N) \to o_N \cong N
$$
with fiber at $q \in N$ given by
$$
\CP(T^*N;o_N,T_q^*N):= \{\gamma:[0,1] \to T^*N \mid \gamma(0) \in o_N, \, \gamma(1) \in
T_q^*N \}.
$$
For given $x \in L_H$, we denote the Hamiltonian trajectory
$$
z_x^H(t) = \phi_H^t((\phi_H^1)^{-1}(x))
$$
which is a Hamiltonian trajectory such that, by definition,
\be\label{eq:zxHat0}
z_x^H(0) \in o_N, \quad z_x^H(1) = x.
\ee
We denote $L_H = \phi_{H}^1(o_N)$ and by $i_H:L_H \hookrightarrow T^*N$ the inclusion map.

Motivated by Weinstein's observation that the action functional
$$
\CA^{cl}_H: \CP(T^*N;o_N) \to \R
$$
can be interpreted as the canonical generating function of $L_H$, the present author
constructed a family of spectral invariants of $L_H$
by performing a mini-max theory via the chain level Floer homology theory in
\cite{oh:jdg,oh:cag}. Indeed, the function defined by
\be\label{eq:hH}
h_H(x) = \CA^{cl}_H(z_x^H)
\ee
is a canonical generating function of $L_H$ in that
\be\label{eq:i*theta}
i_H^*\theta = dh_H.
\ee
We call $h_H$ the \emph{basic generating
function} of $L_H$.  As a function on $N$, not on $L_H$, it
is a multi-valued function. Similarly, one may regard $N \to \phi_H^1(o_N)$ as
a multi-valued section of $T^*N$.

By considering the moduli space of
solutions of the perturbed Cauchy-Riemann equation
\be\label{eq:CRHJq}
\begin{cases}
\dudtau + J\left(\dudt - X_H(u) \right) = 0 \\
u(\tau,0), \, u(\tau,1) \in  o_N,
\end{cases}
\ee
and applying a chain-level Floer mini-max theory, the author \cite{oh:cag}
defined a homologically essential critical value, denoted by
$\rho(H;a)$ associated to each cohomology class $a \in H^*(N)$.
(A similar construction using the
generating function method was earlier given by Viterbo \cite{viterbo:generating}
and it is shown in \cite{milinko, milinko-oh} that both invariants
coincide \emph{modulo a normalization constant}.)
The number $\rho(H;a)$ depends on $H$, not just on $L_H = \phi_H^1(o_N)$.

\subsection{Statement of main results}
\label{subsec:mainresults}

We will be particularly interested in the two spectral invariants $\rho^{lag}(F;1)$,
$\rho^{lag}(F;[pt]^\#)$ and their difference
$\rho^{lag}(F;1) - \rho^{lag}(F;[pt]^\#)$.
This difference does not depend on the
choice of normalization mentioned above. Therefore  we can define a function
$$
\gamma^{lag}: \frak{Iso}(o_N;T^*N) \to \R
$$
unambiguously by setting
\be\label{eq:gammalag}
\gamma^{lag}(L;o_N): =\rho^{lag}(F;1) - \rho(F;[pt]^\#)
\ee
for $L = \phi_F^1(o_N)$. We call this function the spectral capacity
of $L$ (relative to the zero section $o_N$). (See \cite{viterbo:generating}, \cite{oh:cag}.)

We denote by $\gamma^{lag}_B$ the restriction of $\gamma^{lag}$ to the subset
$\frak{Iso}_B(o_N;T^*N)$.
The following Hamiltonian continuity result is the Lagrangian analog to
Corollary 1.2 of \cite{seyfad}.

\begin{thm}[Theorem \ref{thm:rhoC0conti}]\label{thm:continuity}
Let $N$ be a closed manifold. Then the function $\gamma^{lag}_B$ is continuous on $\frak{Iso}_B(o_N;T^*N)$
with respect to  the Hamiltonian $C^0$-topology defined above.
\end{thm}

The following is a very interesting open question
on the Hamiltonian $C^0$-topology.

\begin{ques}\label{ques:ham-continuity} Is the full function $\gamma^{lag}: \frak{Iso}(o_N;T^*N) \to \R$
continuous (without restricting to $\frak{Iso}_B(o_N;T^*N)$ with $B$ having non-empty interior)?
\end{ques}

The question seems to be an important matter to understand
in $C^0$ symplectic topology. Indeed the affirmative answer to the
question is a key ingredient in relation to Viterbo's symplectic homogenization program \cite{viterbo:homogenize}.
The quesiton is sometimes called Viterbo's conjecture. We refer to Theorem \ref{thm:capacity} for the more precise statement on the
relationship between the Hamiltonian $C^0$-distance $d^{ham}_{C^0}$
and the spectral capacity $\gamma^{lag}_B(\phi_F^1(o_N))$ and
the support conditions (2), (3) of the Hamiltonian path $\phi_{F}$ given in
Definition \ref{defn:hamC0onB}.

To properly handle the individual number $\rho^{lag}(F;1)$, not just the difference of
$\rho^{lag}(F;1)$ and $\rho^{lag}(F;[pt]^\#)$, and relate it to
the Lagrangian submanifold $L_F = \phi_F^1(o_N)$ itself, not to the function $F$, we need to put an
additional normalization condition relative to $L_F$. In this regard, it is
useful to take the point of view of weighted Lagrangian submanifolds $(L,\rho_N)$
introduced in \cite{alan:isodrast}, where $\rho_N$ is a probability density on $N$.
Using this $\rho_N$, we can put a normalization condition with respect to the
chosen measure
which is the Lagrangian analog to the mean-normalization of Hamiltonians
$$
\int_M F(t,x)\, \omega^n = 0.
$$
The next result concerns an enhancement of the construction of
basic phase function $f_H$ carried out in \cite{oh:jdg} in the level of topological
Lagrangian embedding. This is a graph selector constructed via Lagrangian Floer
homology. Then the map $\sigma_F: N\setminus \Sing(\sigma_F) \to T^*N$ defined by
$$
\sigma_F(q) : = df_F(q)
$$
selects a single valued branch of $\phi_F^1(o_N)$ on the open subset
$N\setminus \Sing(\sigma_F)$ of full measure when we regard $\phi_F^1(o_N)$ as a
multi-valued section $N \to T^*N$. We call $\sigma_F$
basic Lagrangian selector of $\phi_F^1(o_N)$.
In turn the pair $(\sigma_F,f_F)$ selects a
single valued branch of the wave front of $\phi_F^1(o_N)$ which lies in the
one-jet space $J^1(N) \cong T^*N \times \R$.

\begin{thm}[Corollary \ref{cor:fHiconv}]
Suppose $\phi_{F_i} \to \phi_F$ in the weak Hamiltonian topology given in
Definition \ref{defn:hamtopology} and $L_i = \phi_{F_i}^1(o_N)$.
Then $(\sigma_{F_i},f_{F_i})$ converges uniformly in $J^1(N)$, whose limit
defines a single-valued continuous section of $J^1(M)$ on $N \setminus \Sing(\sigma_F)$.
\end{thm}

Here we define
$$
\Sing(\sigma_F) : = \{q \in N \mid f_F \, \text{ is not differentiable at $q$ }\}
$$
and call it the singular locus of $f_F$. It follows from definition that
$\Sing(\sigma_F)$ is a subset of the so called \emph{Maxwell set} of the
Lagrangian projection $\phi_F^1(o_N) \to N$. (See \cite{givental:singular,arnold:wavefront,zak-roberts}
for detailed study of the Maxwell set.)

We first note that for a generic choice of $F$, $\Sing(\sigma_F)$ is decomposed into
the union of smooth manifolds
$$
\Sing(\sigma_F) = \bigcup_{k=1}^n S_k(\sigma_F)
$$
where $S_k(\sigma_F)$ is the stratum of codimension $k$ in $N$. Along each
connected component of the codimension one strata $S_1(\sigma_F)$, $\Sigma_F$
has two branches. We denote by $f_F^\pm$ the restrictions of $f_F$ in a
neighborhood of the component in each branch respectively.

The next theorem concerns the structure of $\Sing(\sigma_F)$ in the micro-local level.

\begin{thm}[Theorem \ref{thm:conormal}] Let $q \in S_1(F)$.
Then
$$
df^-_{F}(q) - df^+_{F}(q) \in T_q^*N,
$$
which is contained in the conormal space $\nu_q^*[S_1(\sigma_{F});N] \subset T_q^*N$.
\end{thm}

In dimension 2, a complete description of generic singularities of the Lagrangian projection
is available (see  \cite{givental:singular,arnold:wavefront,zak-roberts} for the precise
statement). Based on this generic description of the singulariies, we can
precisely define the notion of \emph{cliff-wall surgery} in dimension 2, which replaces the
multi-valued graph $\phi_F^1(o_N)$ by a rectifiable Lagrangian cycle.
A finer structure theorem is needed to perform similar
surgery in higher dimension which will be studied elsewhere.
It appears to the author that these results seem to carry
some significance in relation to $C^0$-symplectic topology and Hamiltonian dynamics,
which may be worthwhile to pursue further in the future.

Finally we prove the following inequality between the basic phase
function and the Lagrangian spectral invariants.
The inequality stated in this theorem is closely related to
Proposition 5.1 of \cite{viterbo:generating}, whose statement
and proof were formulated in terms of the generating function.

\begin{thm}[Theorem \ref{thm:rhoversusfH}] For any Hamiltonian $F = F(t,x)$, we have
$$
\rho^{lag}(F;[pt]^\#) \leq \min f_F, \quad \max f_F \leq \rho^{lag}(F;1).
$$
\end{thm}

The proof of the second inequality uses a judicious usage of
the triangle product in Lagrangian Floer homology \cite{oh:cag,seidel:triangle,fooo:book}
after a careful consideration of
normalization problem in section \ref{subsec:assign}. We would like to emphasize that the issue of normalization
problem concerning $\rho^{lag}(F;1)$ is a delicate one when one would like to
regard $\rho^{lag}(F;1)$ as an invariant attached to the Lagrangian submanifold itself,
not just to the Hamiltonian $F$. Once the second inequality is established, the first
one easily follows from this and the behavior of spectral invariants $\rho^{lag}(\cdot;\{q\})$
under the duality map $F \mapsto F^\frak r(t,x) = - F(t,\frak r(x))$ induced by
the anti-symplectic reflection $\frak r: T^*N \to T^*N, \, \frak r(q,p) = (q, -p)$ for $x =(q,p)$ similarly as done in
\cite{oh:cag} for the duality map $F \mapsto \widetilde F(t,x) = - F(1-t,x)$.
(We thank Seyfaddini for pointing out to us \cite{seyfad:e-mail} that the first inequality should
also hold in the presence of the second inequality in Theorem \ref{thm:rhoversusfH}.)
See also \cite{viterbo:generating} for the similar consideration of this reflection map in the
context of generating function techniques.

The research performed in this paper is partially motivated by the study of
topological Hamiltonian dynamics and its applications to the problem of simpleness
question on the area-preserving homeomorphism group of the 2-disc. We anticipate
that these
studies play some important role in the study of homotopy invariance of Hamiltonian spectral
invariant function $\phi_F \mapsto \rho(F;a)$ for a topological Hamiltonian path $\phi_F$
in the sense of \cite{oh:hameo1,oh:hameo2} on any closed symplectic manifolds $(M,\omega)$.
It should also be regarded as
a natural continuation of the author's study of Lagrangian spectral invariants
performed in \cite{oh:jdg,oh:cag}.

We thank F. Zapolsky for attracting our attention to
the preprint \cite{MVZ} from which we have learned the Lagrangian version of
the optimal triangle inequality, and S. Seyfaddini for sending us his
very interesting preprint \cite{seyfad} before its publication, which greatly helps us
in proving the Hamiltonian continuity of Lagrangian spectral capacity.
We also thank A. Givental for many enlightening e-mail
communications concerning the structure of Maxwell set,
Proposition \ref{prop:zak-roberts}  and the cliff-wall surgery.
\bigskip

\centerline{\bf Notations and Conventions}
\par\smallskip

We follow the conventions
of \cite{oh:alan,oh:hameo2} for the definition of
Hamiltonian vector fields and action functional,
and others appearing in the Hamiltonian Floer theory and in the construction of spectral
invariants on general closed symplectic manifold. They are
different from e.g., those used in \cite{pol:book,entov-pol:morphism}
one way or the other, but coincide with those used in
\cite{seyfad}.

\begin{enumerate}
\item We usually use the letter $M$ to denote a symplectic manifold and $N$
to denote a general smooth manifold.
\item The Hamiltonian vector field $X_H$ is defined by
$dH = \omega(X_H,\cdot)$.
\item The flow of $X_H$ is denoted by $\phi_H: t \mapsto \phi_H^t$ and
its time-one map by $\phi_H^1 \in Ham(M,\omega)$.
\item We denote by $z^q_H(t) = \phi_H^t(q)$ the Hamiltonian trajectory
associated to the initial point $q$.
\item We denote by $z_x^H(t) = \phi_H^t((\phi_H^1)^{-1}(x))$ the Hamiltonian trajectory
associated to the final point $x$.
\item $\overline H(t,x) = -H(t, \phi_H^t(x))$ is the Hamiltonian generating
the inverse path $(\phi_H^t)^{-1}$.
\item The canonical symplectic form on the cotangent bundle $T^*N$ is denoted by $\omega_0 = - d\theta$
where $\theta$ is the Liouville one-form which is given by $\theta = \sum_i p_i \, dq^i$
in the canonical coordinates $(q^1,\cdots, q^n, p_1, \cdots, p_n)$.
\item The classical Hamilton's action functional on the space of paths in $T^*N$ is given by
$$
\CA^{cl}_H(\gamma) = \int \gamma^*\theta - \int_0^1 H(t,\gamma(t))\, dt.
$$
\item We denote by $o_N$ the zero section of $T^*N$.
\item We denote $\rho^{lag}(H;a)$ the Lagrangian spectral invariant on $T^*N$
(relative to the zero section $o_N$) defined in \cite{oh:jdg} for asymptotically
constant Hamiltonian $H$ on $T^*N$.
\item We denote by $f_H$ the basic phase function and its associated Lagrangian
selector by $\sigma_H: N \to T^*N$ given by $\sigma_H(q) = df_H(q)$ at which $df_H(q)$
exists.
\end{enumerate}

\section{Basic generating function $h_H$ of Lagrangian submanifold}
\label{sec:generating}

In this section, we recall the definition of basic generating function.

Let $H = H(t,x)$ be a Hamiltonian on $T^*N$ which is asymptotically
constant i.e., one whose  Hamiltonian vector field $X_H$ is compactly supported.
Denote by $\CP C^\infty_{asc}(T^*N,\R)$ be the set of such a family of functions.
We denote $L_H = \phi_{H}^1(o_N)$ and
denote by $i_H:L_H \hookrightarrow T^*N$ the inclusion map.

Recall the classical action functional is defined as
$$
\CA^{cl}_H(\gamma) = \int \gamma^*\theta - \int_0^1 H(t,\gamma(t))\, dt
$$
on the space $\CP(T^*N)$ of paths $\gamma:[0,1] \to T^*N$, and its first variation formula
is given by
\be\label{eq:1stvariation}
d\CA^{cl}_H(\gamma)(\xi) = \int_0^1 \omega(\dot\gamma - X_H(t,\gamma(t)), \xi(t))\,dt
- \langle \theta(\gamma(0)),\xi(0) \rangle + \langle \theta(\gamma(1)), \xi(1) \rangle.
\ee
For given $q \in o_N \cong N$, we denote
$$
z_H^q(t) = \phi_H^t(q)
$$
which is a Hamiltonian trajectory such that
\be\label{eq:zqHat0}
z_H^q(0) = q \in o_N,
\ee
which specifies the initial point $q \in o_N$.
(We remark that the notation here is slightly different from that of
\cite{oh:jdg,oh:cag} in that $z_H^q$ therein denotes $z^H_q$ in this paper.
We adopt the current notation to be consistent with that of \cite{oh:book} and
other recent papers of the author.)

We define the function $\widetilde h_H: [0,1] \times N \to \R$ by
\be\label{eq:tildehH}
\widetilde h_H(t,q) = \int \left(z_H^q|_{[0,t]}\right)^*\theta - \int_0^t H(u,\phi_H^u(q))\, du
\ee
call it the space-time (or parametric) basic generating function in the fixed frame.

The following basic lemma follows immediately from \eqref{eq:1stvariation} whose proof we
omit.

\begin{lem}\label{lem:generating}
The function $\widetilde h_H$ satisfies
\bea
d\widetilde h_H(t,q) & = & \left((z_H^q)^*\theta(t) - H(t, z_H^q(t))\, dt\right) +
(\psi_H^t)^*\theta \label{eq:dtildeh}\\
& = & \psi_H^*\theta - H(t, z_H^q(t))\, dt \label{eq:dtildeh2}
\eea
where $\psi_H: [0,1] \times N \to T^*N$ defined by $\psi_H(t,q) =
\phi_H^t|_{o_N}$ and $\psi_H^t(q) = \psi_H(t,q)$.
\end{lem}

It turns out that the following form of Hamiltonian trajectories
\be\label{eq:zxH}
z_x^H(t) = \phi_H^t((\phi_H^1)^{-1}(x))
\ee
are also useful, which specifies the \emph{final point} of the trajectory instead of
the initial point as specified in the trajectory $z^q_H$.
Then we define
\be\label{eq:hH}
h_H(t,x) = \widetilde h_H(t,(\phi_H^t)^{-1}(x)), \quad x \in \phi_H^t(o_N)
\ee
in the \emph{moving frame}.

Now consider the Lagrangian submanifold $\phi_H^1(o_N)$. We would like to point out that the function
$$
h_H(1,\cdot):L_H \to \R\,; \, h_H(1,x):= \widetilde h_H(1,(\phi_H)^{-1}(x))
$$
defines the natural generating function of $L_H : = \phi_H^1(o_N)$ in that $d_x h_H = i_H^*\theta$
where $i_H:L_H \to T^*N$ is the canonical inclusion map. The image of the map
$$
x \in L_{H} \mapsto (h_{H}(x),x)
$$
defines a canonical Legendrian lift of $L_{H}$ in the one-jet bundle
$J^1(N) \cong \R \times T^*N$. We call $h_H$ the basic generating function
in the moving frame.  We denote the corresponding Legendrian submanifold by $R_H$.
However, as a function on $N$, $h_{H}$ is multi-valued,
while $\widetilde h_H$ is a well-defined single-valued function.

In general, the projection $R \to \R \times N$ of any Legendrian submanifold $R \subset J^1(N,\R)
= \R \times T^*N$ is called the wave front \cite{eliash:front} of the Legendrian
submanifold $R$. We denote by $W_R \subset \R \times N$ by the front of $R$.
We also define the (Lagrangian) action spectrum of $H$ on $T^*N$ by
\be\label{eq:SpecL}
\Spec(H;N) = \{\CA^{cl}_H(z_x^H) \mid x \in L_H \cap o_N\}
\ee
which also coincides with the set of critical values of $h_H$.
It follows that $\Spec(H;N)$ is a compact subset of $\R$ of measure zero.

\begin{rem} We would like to note that we have no a priori control of $C^0$ bound
for the functions $h_H$ (or equivalently $\widetilde h_H$), even when $H$ is
bounded in $L^{(1,\infty)}$ norm. Getting this $C^0$-bound is equivalent to getting
the bound for the actions of the relevant Hamiltonian chords. Indeed understanding the
precise relationship between the
action bound, the norm $\|H\|$ and the $C^0$-distance of the time-one map $\phi_H^1$
is a heart of the matter in $C^0$ symplectic topology.

In section \ref{sec:basic}, we recall construction of \emph{basic phase function}
$f_H$ from \cite{oh:jdg} which is a particular single valued selection of
the multivalued function $h_H$ on $N$ that has particularly nice properties
in relation to the study of spectral invariants of the present paper.
This function was constructed via the Floer mini-max arguments similarly as the
spectral invariants $\rho^{ham}(H;a)$ is defined in \cite{oh:jdg}, and its $C_0$-norm is bounded by $\|H\|$.
\end{rem}

\section{Basic phase function and its associated Lagrangian selector}
\label{sec:basic}

In this section, we first recall the definition of basic phase
function constructed in \cite{oh:jdg}. Then we introduce a crucial
measurable map $\varphi^H: N \to N$, which is defined by a selection of
of a single valued branch of the multivalued section
$$
N \to L_H \subset T^*M
$$
followed by $(\phi_H^1)^{-1}$. We call this map the mass transfer map
associated to the Hamiltonian $H$. It is interesting to note that such a
selection process was studied e.g., in the theory of multi-valued
functions, or $Q$-valued functions, in the sense of Almgren
\cite{almgren} in geometric measure theory. In particular, in
\cite{DGT}, existence of such a single valued branch is studied in
the general abstract setting of metric spaces and a finite group
action of isometries. It would be interesting to see whether there
would be any other significant intrusion of the theory of multivalued
functions into the study of symplectic topology.

\subsection{Graph selector of wave fronts}
\label{subsec:graphselector}

The following theorem was proved in \cite{chaperon} and in
\cite{oh:jdg} by the generating function method and by the Floer
theory respectively. (According to \cite{PPS}, the proof
of this theorem was first outlined by Sikorav in Chaperon's
seminar.)

\begin{thm}[Sikorav, Chaperon \cite{chaperon}, Oh \cite{oh:jdg}] Let $L \subset T^*N$
be a Hamiltonian deformation of the zero section $o_N$. Then there
exists a Lipschitz continuous function $f: N \to \R$, which is
smooth on an open subset $N_0 \subset N$ of full measure, such that
$$
(q,df(q)) \in L
$$
for every $q \in N_0$. Moreover if $df(q) = 0$ for all $q \in
N_0$, then $L$ coincides with the zero section $o_N$. The choice of
$f$ is unique modulo the shift by a constant.
\end{thm}

The details of the proof of Lipschitz continuity of $f$ is given
in \cite{PPS}. We denote by $\operatorname{Sing} f$
the set of non-differentiable points of $f$. Then by definition
$$
N_0=\operatorname{Reg} f: = N \setminus \operatorname{Sing} f
$$
is a subset of full measure and $f$ is differentiable thereon.

We call such a function $f$ a graph selector in general following the terminology of \cite{PPS}
and denote the corresponding graph part of the front of the Legendrian submanifold $R$ by
$$
G_{f}: = \{(h_L(q,df(q)), q,df(q)) \mid q \in N \} \subset R.
$$
By construction, the projection $\pi_R: G_{f} \to N$ restricts to a one-one
correspondence and the function $f: \operatorname{Reg} f
\to \R$ continuously extends to $\overline{\operatorname{Reg} f} = N$.

By definition,
\be\label{eq:dfL} |df(q)| \leq \max_{x \in L}
|p(x)|
\ee
for any $q \in N_0$, where $x = (q(x),p(x))$ and the norm
$|p(x)|$ is measured by any given Riemannian metric on $N$.

\begin{prop} As $d_{\text{\rm H}}(L,o_N) \to 0$, $|df(q)| \to 0$ uniformly over $q \in N_0$.
\end{prop}

In \cite{oh:jdg}, a canonical choice of $f$ is constructed via the
chain level Floer theory, \emph{provided} the generating Hamiltonian
$H$ of $L$ is given. The author called the corresponding graph selector $f$
the basic phase function of $L = \phi_H^1(o_N)$ and denoted it by $f_H$.
We give a quick outline of the construction referring the
readers to \cite{oh:jdg} for the full details of the construction.

\subsection{The basic phase function $f_H$ and its Lagrangian selector}
\label{subsec:basic}

Another construction in \cite{oh:jdg} is given by considering the
Lagrangian pair
$$
(o_N, T^*_qN), \quad q \in N
$$
and its associated Floer complex $CF(H;o_N, T^*_qN)$ generated by
the Hamiltonian trajectory $z:[0,1] \to T^*N$ satisfying
\be\label{eq:Hamchordeq2}
\dot z = X_H(t,z(t)), \quad z(0) \in o_N, \, z(1) \in T^*_qN.
\ee
Denote by $\CC hord(H;o_N,T^*_qN)$ the set of
solutions. The differential $\del_{(H,J)}$ on $CF(H;o_N, T^*_qN)$ is
provided by the moduli space of solutions of the perturbed
Cauchy-Riemann equation
\be\label{eq:CRHJqN}
\begin{cases}
\dudtau + J\left(\dudt - X_H(u) \right) = 0 \\
u(\tau,0) \in o_N, \, u(\tau,1) \in T^*_qN.
\end{cases}
\ee

An element $\alpha \in CF(H;o_N,T^*_qN)$ is expressed as a finite
sum
$$
\alpha = \sum_{z \in \CC hord(H;o_N,T_q^*N)} a_z [z], \quad a_z \in
\Z.
$$
We denote the level of the chain $\alpha$ by
$$
\lambda_H(\alpha): = \max_{z \in \supp \alpha} \{\CA^{cl}_H(z)\}.
$$
The resulting invariant $\rho^{lag}(H;\{q\})$ is to be defined by the mini-max
value
$$
\rho^{lag}(H;\{q\}) = \inf_{\alpha \in [q]}\lambda_H(\alpha)
$$
where $[q] \in H_0(\{q\};\Z)$ is a generator of the homology group
$H_0(\{q\};\Z)$.

 A priori, $\rho^{lag}(H;\{q\})$ is defined when
$\phi_H^1(o_N)$ intersects $T_qN^*$ transversely but can be extended
to non-transversal $q$'s by continuity. By varying $q \in N$, this
defines a function $f_H: N \to \R$ which is precisely the one called the basic
phase function in \cite{oh:jdg}. (A similar construction of such a function using the generating function
method was earlier given by Sikorav and Chaperon \cite{chaperon}.)
We call the associated
graph part $G_{f_H}$ the basic branch of the front $W_{R_H}$ of $R_H$.

\begin{thm}[\cite{oh:jdg,oh:alan}]
There exists a solution $z:[0,1] \to T^*N$ of $\dot z = X(t,z)$ such
that $z(0)=q, \, z(1) \in o_N$ and $\CA^{cl}_{H}(z) = \rho^{lag}(H;\{q\})$
whether or not $\phi_H^1(o_N)$ intersects $T_q^*N$
transversely.
\end{thm}

We summarize the main properties of $f_H$ established in
\cite{oh:jdg}.

\begin{thm}[\cite{oh:jdg}]\label{thm:ohcag}
When the Hamiltonian $H=H(t,x)$ such that $L = \phi_H^1(o_N)$ is
given, there is a canonical lift $f_H$ defined by $f_H(q): =
\rho^{lag}(H;\{q\})$ that satisfies \be\label{eq:fversush} f_H \circ
\pi(x) = h_H(x) = \CA^{cl}_H(z_x^H) \ee for some Hamiltonian chord
$z_x^H$ ending at $x \in T^*_qN$. This $f_H$ satisfies the following
property in addition \be\label{eq:fH} \|f_H - f_{K}\|_\infty \leq
\|H - K\|. \ee
\end{thm}

An immediate corollary of Theorem  is

\begin{cor}\label{cor:fHiconv} If $H_i$ converges in $L^{(1,\infty)}$,
then $f_{H_i}$ converges uniformly.
\end{cor}

Based on this corollary, we will just denote the limit continuous function by
\be\label{eq:limitfH}
f_H: = \lim_{i \to \infty}f_{H_i}
\ee
when $H_i \to H$ in $L^{(1,\infty)}$-topology, and call it the
basic phase function of the topological Hamiltonian $H$ or of the
$C^0$-Lagrangian submanifold $L_H = \phi_H^1(o_N)$.

Note that $\pi_H= \pi|_{L_H}: L_H = \phi_H^1(o_N) \to N$ is surjective for all $H$
(see \cite{lal-sik} for its proof) and so
$\pi_H^{-1}(\pi_H^{-1}(q)) \subset o_N$ is a non-empty
compact subset of $o_N \cong N$. Therefore we can regard the
`inverse' $\pi_H^{-1}:N \to L_H \subset T^*N$
as a everywhere defined multivalued section of $\pi: T^*N \to N$.

We introduce the following general definition

\begin{defn} Let $L \subset T^*N$ be a Lagrangian submanifold projecting
surjectively to $N$. We call a single valued section $\sigma$ of $T^*N$ with values
lying in $L$ a \emph{Lagrangian selector} of $L$.
\end{defn}

For any given Lagrangian selector $\sigma$ of $L = L_H =
\phi_H^1(o_N)$, we define the map $\varphi^\sigma: N \to N$ to be
$$
\varphi^\sigma(q) = (\phi_H^1)^{-1}(\sigma(q)).
$$
Recall that the graph $G_{f_H}$ is a subset of the front $W_{R_H}$
of $R_H$ and for a generic choice of $H$ the set $\operatorname{Sing}f_H \subset N$
consists of the crossing points of the two different branches and the cusp points
of the front of $W_{R_H}$. Therefore it is a set of measure zero in $N$.
(See \cite{eliash:front}, \cite{PPS}, for example.)
Once the graph selector $f_H$ of $L_H$ is picked out, it provides a
natural Lagrangian selector defined by
$$
\sigma_H(q): = \text{Choice}\{x \in L_H \mid \pi(x) = q, \,  \CA^{cl}_H(z^H_x) = f_H(q)\}
$$
via the axiom of choice where $\text{Choice}$ is a choice function. It satisfies
\be\label{eq:sigmaHdfH}
\sigma_H(q) = df_H(q)
\ee
whenever $df_H(q)$ is defined. We call this particular Lagrangian selector of
$L_H$ the basic Lagrangian selector and the pair $(\sigma_H,f_H)$ the
basic wave front of the Lagrangian submanifold $\phi_H^1(o_N)$.

The general structure theorem of the wave front (see
\cite{eliash:front}, \cite{PPS} for example) proves that the section
$\sigma_H$ is a differentiable map on a set of full measure for a
generic choice of $H$ which is, however, \emph{not necessarily
continuous}: This is because as long as $q \in N \setminus
\operatorname{Sing}f_H$, we can choose a small open neighborhood of
$U \subset N \setminus \operatorname{Sing}f_H$ of $q$ and $V \subset
L_H = \phi_H^1(o_N)$ of $x \in V$ with $\pi(x) = q$ so that the
projection $\pi|_V: V \to U$ is a diffeomorphism.

Then we define the mass transfer map $\varphi^H: N \to N$ by
\be\label{eq:qH}
\varphi^H(q) = (\phi_H^1)^{-1}(\sigma_H(q)).
\ee
The map $\varphi^H$ is \emph{measurable, but not necessarily continuous}, which is
however differentiable on a set of full measure for a generic choice of $H$.
And from its definition, it is surjective if and only if the Lagrangian submanifold
$\phi_H^1(o_N)$ is a graph of an exact one-form.
On the other hand, the map $\varphi^H$
may not be continuous along the subset $\operatorname{Sing}f_H \subset N$ which is
a set of measure zero. By definition, we have
\be\label{eq:fHqAAH}
f_H(q) = \CA^{cl}_H\left(z_H^{\varphi^H(q)}\right) = \widetilde h_H(\varphi^H(q)).
\ee
This relationship between $f_H$ and $\widetilde h_H$ is the reason
why we introduce the transfer map $\varphi^H$.

The following lemma is obvious from the definition of $\varphi^H$. We note
$$
d_{\text{\rm H}}(\phi_H^1(o_N),o_N) \leq \osc_{C^0}(\phi_H^1;o_N)
$$
where $d_{\text{\rm H}}(\phi_H^1(o_N),o_N)$ is the Hausdorff distance.

\begin{lem}\label{lem:varphiKi1to0} We have
$$
d(\varphi^H(x),x) \leq d_{\text{\rm H}}(\phi_H^1(o_N),o_N) + \osc_{C^0}(\phi_H^1;o_N)
\leq 2 \osc_{C^0}(\phi_H^1;o_N)
$$
for all $x \in N_0$. In particular, if $\osc_{C^0}(\phi_H^1;o_N) \to 0$, then
$\max_{x \in N_0} d(\varphi^H(x),x) \to 0$ uniformly over $x \in N_0$.
\end{lem}

\section{Singular locus of the basic phase function and cliff-wall surgery}
\label{sec:conormal}

We first recall two important properties of the Liouville one-form $\theta$:
\begin{enumerate}
\item $\theta$ identically vanishes on any conormal variety.
(See \cite{oh:jdg,kasturi-oh1} for the explanation
on the importance of this fact in relation to the Lagrangian Floer theory on the cotangent bundle.)
\item  For any one form $\alpha$ on $N$, we have $\widehat \alpha^*\theta = \alpha$
where $\widehat \alpha: N \to T^*N$ is the section map associated to the one-form $\alpha$
as a section of $T^*N$. In particular, we have
$$
\sigma_{F}^*\theta = df_{F}
$$
on $N \setminus \operatorname{Sing}(\sigma_{F})$ and on each stratum of
$\operatorname{Sing}(\sigma_{F})$.
\end{enumerate}

We note that the singular locus $S(\sigma_F) \subset \Delta$ is a subset of
the \emph{bifurcation diagram} of the Lagrangian submanifold $\phi_F^1(o_N)$:
The bifurcation diagram is the union of the caustic and the Maxwell set
where the latter is the set of points of which merge the different branches of the
generating function $h$. (See section 4 \cite{givental:singular} for
the definition of bifurcation diagram of Lagrangian submanifold $L \subset T^*N$
in general.)

For a generic $F$, $S(\sigma_F)$ is stratified
into a finite union of smooth submanifolds
$$
\bigcup_{k=1}^{n} S_k(\sigma_F), \quad S_k(\sigma_F)= \Sing_k(\sigma_F), \quad n = \dim N
$$
(see \cite{arnold:normalform,eliash:front,givental:singular} e.g., for such a result)
so that its conormal variety $\nu^*S(\sigma_{F})$ can be defined as a finite union of
conormals of the corresponding strata.
Each stratum $\Sing_k(\sigma_{F})$ has codimension $k$ in $\Delta$.
The stratum for some $k$ could be empty. (See \cite{kash-schapira}. See also
\cite{kasturi,kasturi-oh2}, \cite{nadler-zaslow,nadler} for the usages of such conormal varieties
in relation to Lagrangian Floer theory.)

In $\dim N = 2$, there are two strata to consider,
one $S_1(\sigma_F)$ and the other $S_2(\sigma_F)$.

For $k=1$, each given point $q \in S_1(\sigma_{F})$ has a neighborhood
$A(q) \subset N$ such that $A(q) \setminus S_1(\sigma_{F})$
has two components.
We also note that $\Sigma_F$ carries a natural orientation induced from $N$
by projection when $N$ is orientable and so defines an integral current
in the sense of geometric measure theory \cite{federer}.
When $N$ is oriented, $S_1(F)$ is also
orientable as a finite union of smooth hypersurface. We fix any orientation
on $S_1(F)$.

We denote by $A^\pm(q)$ the closure of
each component of $A(q) \setminus S_1(\sigma_{F})$ in $A(q)$ respectively.
Here we denote by $A^+(q)$ the component whose boundary orientation
on $\del A^+(q)$ coincides with that of the given orientation on $S_1(F)$
and by $\del A^-(q)$ the other one.
Then each of $A^\pm(q)$ is an open-closed domain with the same boundary
$$
\del A^\pm(q) = A(q) \cap S_1(\sigma_{F}).
$$
Denote
\be\label{eq:fpmatq}
df^\pm_{F}(q) = \lim_{p_\pm \to q} df_{F}(p_\pm)
\ee
obtained by taking the limit on $A^\pm(q)$ respectively. The limits
are well-defined from the definition of $\sigma_F$
since $\Im \sigma_F= \Im \widehat{df_F} \subset \phi_F^1(o_N)$ where $\phi_F^1(o_N)$
is a smooth closed submanifold in $T^*N$.

We now prove the following theorem. We refer to
\cite{givental:singular}, \cite{zak-roberts} for a related statement.

\begin{thm}\label{thm:conormal} Let $q \in S_1(F)$.
Then
$$
df^-_{F}(q) - df^+_{F}(q) \in T_q^*N,
$$
which is contained in the conormal space $\nu_q^*[S_1(\sigma_{F});N] \subset T_q^*N$.
\end{thm}
\begin{proof} Let $\vec v \in T_qS_1(\sigma_{F})$ be any given
tangent vector.
Choose a smooth curve $\gamma: (-\e,\e) \to S_1(\sigma_{F})$
with $\gamma(0) = q$. For any given sufficiently small $\delta \geq 0$,
we define a family of $\delta$-shifted curves
$$
\gamma^\pm_{\delta}(t) = \exp_{\gamma(t)}(\pm \delta \vec n(t)),
$$
where $\exp$ is the normal exponential map of $S_1(\sigma_{F})$
in $N$ and $\vec n(t)$ is the unit normal vector thereof at $\gamma(t)$ towards the domain $A^+(q)$.
Then $\gamma^+_{\delta}$ is mapped into $\Int A^+(q)$ and
$\gamma^-_\delta$ into $\Int A^-(q)$ for all sufficiently small $\delta > 0$.
Note
$$
\gamma^\pm_{0}(t) = \gamma(t)
$$
for $\delta = 0$.
Since $f_{F}: N \to \R$ is a continuous function, we have the uniform
convergence
$$
f_{F}(\gamma^+_{\delta}(t)) - f_{F}(\gamma^-_{\delta}(t)) \to 0
$$
as $\delta \to 0$ over $t \in (-\e,\e)$. Furthermore since $f_F$ is smooth up to the boundary on
each of $A^\pm(q)$ and $df_F$ is uniformly differentiable up to the boundary of $A^\pm(q)$
for either of $\pm$,
\beastar
f_{F}(\gamma^\pm_{\delta}(t))
& = & f_{F}(\gamma^\pm_{\delta}(0)) + t\, df_{F}(\gamma^\pm_{\delta}(0))
((\gamma^\pm_{\delta})'(0)) + O(|t|^2)\\
& = & f_{F}(\gamma^\pm_{\delta}(0))
+ t\, df_{F}(\gamma^\pm_{\delta}(0))\circ D\exp_{\gamma(0)}(\pm \delta \vec n(0))(\gamma'(0)) + O(|t|^2)
\eeastar
where $|O(|t|^2)| \leq C |t|^2$ for a constant $C> 0$ uniformly over $\delta \geq 0$
and $t \in (-\e, \e)$. Here $D\exp_p(\vec n) (\vec v)$ is the derivative
$$
D\exp_p(\vec n) (\vec v): = \frac{d}{dt}\Big|_{t=0} \exp_{\gamma(t)}(\vec n), \, \vec v
= \gamma'(0), \, \gamma(0) = p,
$$
which is nothing but the covariant derivative of the Jacobi field along the geodesic
$t \mapsto \exp_p(t v)$ with the initial vector $\vec n$ at $p$. (See \cite{kacher} for
an elegant exposition on the detailed study of exponential maps.)
By letting $\delta \to 0$ and using the uniformity of the constant $C$ and the continuity
of $f_{F}$, we obtain
\beastar
f_{F}(\gamma(t)) & = & f_{F}(q) + \lim_{\delta \to 0}(t\, df_{F}(\gamma^\pm_{\delta}(0))
((\gamma^\pm_{\delta})'(0))) + O(|t|^2)\\
& = & f_{F}(q) + t \, \lim_{\delta \to 0}df_{F}^\pm(\gamma^\pm_{\delta}(0))\left(D\exp_{\gamma(0)}
(\pm \delta \vec n(0))((\gamma^\pm_{\delta})'(0))\right)) + O(|t|^2).
\eeastar
Then by taking the difference of two equations for $\pm$ and dividing by $t$,
utilizing the convergence $(\gamma^\pm_{\delta})'(0) \to \gamma'(0)$ as $\delta \to 0$ and then
evaluating at $t = 0$, we obtain
\beastar
0 = \lim_{\delta \to 0} \left(df_{F}^+(\gamma^+_{\delta}(0))\circ D\exp_{\gamma(0)}(+\delta \vec n(0)) -
d(f_{F}^-(\gamma^-_{\delta}(0))\circ D\exp_{\gamma(0)}(-\delta \vec n(0))\right)(\gamma'(0)).
\eeastar

Recall that $\gamma(0) = p$ and $\gamma^\pm_{\delta}(0) \to p$,
and $D\exp_p(\pm \delta \vec n(0))$ converges to $D\exp_p(\vec 0)$ as $\delta \to 0$, which
is nothing but the identity map on $\nu_q S_1(\sigma_{F})$ by the standard fact on the exponential
map (see \cite{kacher}). Therefore from this last equality, we derive
$$
\left(df_{F}^+(q) - df_{F}^-(q)\right)(\vec v) = 0
$$
by the definition of $df_{F}^\pm(q)$.
Since this holds for all $\vec v \in T_qS_1(\sigma_{F})$,
the proposition for $k =1$ is proved.
\end{proof}

The boundary orientations
of the two components arising from that of $\Sigma_F$, which in turn
is induced from that of $N$ via $\pi_1$ have opposite orientations. We call
the one whose projection to $S_1(\sigma_F)$ under $\pi_1$ coinciding with the given orientation
the \emph{upper branch} and the one with the opposite one the \emph{lower branch}
and denote them by
$$
\del^+\Sigma_F, \, \del^-\Sigma_F
$$
respectively.

Now let $L_q$ be the line segment connecting the two vectors
$df^\pm_{F}(q)$, i.e.,
\be\label{eq:Lq}
L_{q}: u \in [0,1] \mapsto  df^+_{F}(q) + u(df^-_{F}(q) -df^+_{F}(q)) \subset T_q^*N.
\ee
This is an affine line that is parallel to the conormal space $\nu^*_{q}S_1(\sigma_F)$.
Therefore the union
\be\label{eq:unionLq}
\Sigma_{F;[-+]}: = \bigcup_{q \in S_1(\sigma_F)} L_{q}
\ee
is contained in the translated conormal
\be\label{eq:trans-conormal+}
df_F^+ + \nu^*[S_1(\sigma_F);N]
\ee
Here the bracket $[-+]$ stands for the line segment $L_q$, and
 $\nu^*[S_1(\sigma_F);N]$ is the conormal bundle of $S_1(\sigma_F)$ in $N$.
We would like to point out that since $df_F^+(q) - df_F^-(q) \in \nu^*[S_1(\sigma_F);N]$
we have the equality
$$
df_F^+(q) + \nu_q^*[S_1(\sigma_F);N] = df_F^-(q) + \nu_q^*[S_1(\sigma_F);N]
$$
for all $q \in S_1(\sigma_F)$.
Therefore we can simply write \eqref{eq:trans-conormal+} as
\be\label{eq:trans-conormal}
df_F + \nu^*[S_1(\sigma_F);N]
\ee
unambiguously.

\begin{defn}[Basic Lagrangian selector chain]
We denote by $\sigma_F$ the chain whose support is given by
\be\label{eq:SigmaF}
\supp(\sigma_F): = \overline \Sigma_F
\ee
with the orientation given as above, and define its \emph{micro-support} by
\be\label{eq:microsupport}
SS(\sigma_F): = \overline{df_F + \nu^*[S_1(\Sigma_F);N]}
\ee
imitating the notation from \cite{kash-schapira}.
\end{defn}

The two components of $\del \sigma_{F}$ associated to each connected
component of $S_1(\sigma_F)$ are the graphs of $df_F^\pm$ for the functions
$f_F^\pm$ near $S_1(\sigma_F)$.

Note that each connected component of $S_1(\sigma_F)$ gives rise to
two components of $\del \sigma_{F;[-+]} \cap \sigma_F$.
We can bridge the `cliff' between the two branches
of $\del \sigma_{F}$ over each connected component of $S_1(\sigma_{F})$ and

\begin{defn}[Cliff wall chain] We
define a `cliff wall' chain $\sigma_{F;[-+]}$ whose support is given by the union
$$
\Sigma_{F;[-+]} = \bigcup_{q \in S_1(\sigma_{F})} L_q
$$
Then we define the chain
$\sigma_{F;[-+]}$ similarly as we define $\sigma_F$ by taking its closure in $T^*N$.
\end{defn}
We emphasize that $\sigma_{F;[-+]}$ lies outside the Lagrangian submanifold $\phi_F^1(o_N)$.

By definition, its tangent space at $x = (q,u)$ has natural identification with
$$
T_x \Sigma_{F;[-+]} \cong \nu^*_q S_1(\sigma_F) \oplus T_q S_1(\sigma_F).
$$
Due to Theorem \ref{thm:conormal}, it carries a natural direct sum orientation
$$
o_{\Sigma_{F;[-+]}}(q) = \{df^-_F(q) - df^+_F(q)\} \oplus o_{S_1(\sigma_F)}(q).
$$
Therefore $\Sigma_{F;[-+]}$ carries a natural orientation and defines a current.
Under the natural identification of $T_qN$ with $T_q^*N$ by the dual pairing,
which induces an identification
$$
\nu^*_q S_1(\sigma_F) \oplus T_q S_1(\sigma_F) \cong \nu_q S_1(\sigma_F) \oplus T_q S_1(\sigma_F)
$$
as an oriented vector space. Then we have the relation
\be\label{eq:orient}
\del \Sigma_F = - \del \Sigma_{F;[-+]}
\ee
along the intersection $\del \Sigma_F \cap \del \Sigma_{F;[-+]}$.
\begin{rem}
\begin{enumerate}
\item We would like to note that the singular locus $S(\sigma_F) \subset \Delta$ is a subset of
the bifurcation diagram of the Lagrangian submanifold $\phi_F^1(o_N)$:
The bifurcation diagram is the union of the caustic and the Maxwell set
where the latter is the set of points of which merge the different branches of the
generating function $h$. (See section 4 \cite{givental:singular} for
the definition of bifurcation diagram of Lagrangian submanifold $L \subset T^*N$
in general.) But this detailed structure does not play any role in our proof
except the one described.
\item
However we would like to note that each fiber of $SS(\sigma_F)$ is
an affine space
$$
df_F(q) + \nu_q^*[S_1(\Sigma_F);N]
$$
at $q \in S_1(\Sigma_F)$, not a linear space. In fact, if we incorporate the
orientation into consideration, one can refine this definition further
to the `half space' instead of the full affine space. We denote this refinement
by $SS^+(\sigma_F)$. Then at a point $q$ in the lower dimensional
strata, it will be a `wedge domain', i.e., the intersection of several
space of this type. (See \cite{kasturi-oh1,kasturi-oh2} for a usage of such
domains in their quantization program of Eilenberg-Steenrod axiom.)
We will come back to further discussion on the detailed structure of singularities elsewhere.
\end{enumerate}
\end{rem}

Next we consider the case of $S_2(\sigma_F)$ and its relationship with $\sigma_F$ and $S_1(\sigma_F)$.
Note that for a generic choice of $F$, $S_2(\sigma_F)$ consists of a finite number of points in $N$
consisting of either a caustic point or a triple intersection point of
the Maxwell set (see \cite{arnold:normalform},
section 4 \cite{givental:singular} and 7.1 \cite{zak-roberts}).

The following proposition can be also derived from the general structure
theorem of generic singularities of Lagrangian maps.
We restrict the proposition to $\dim N =2$ here postponing the precise statement
for the high dimensional cases elsewhere.

\begin{prop}\label{prop:zak-roberts}
Assume $\dim N =2$.
For a generic choice of $F$, the boundary of $\sigma_F + \sigma_{F;[-+]}$
is a finite union of triangles each of which is formed by the three line segments
$L_q$ given in \eqref{eq:Lq} associated to a
triple intersection point $q$ of $S(\sigma_F)$ contained in $S_2(\sigma_F)$.
Furthermore each triangle is the boundary of a 2-simplex
contained in the fiber $T_q^*N$.
\end{prop}
\begin{proof} This is an immediate consequence of the classification
theorem of generic singularities in dimension 2 of Lagrangian maps
originally proved by Arnold \cite{arnold:normalform}.
(See also p. 55 and Figure 43 \cite{arnold:wavefront},
section 4 \cite{givental:singular} and section 7.1 \cite{zak-roberts}.)
\end{proof}

Now we define $\sigma_{F;\Delta^2}$ to be the union of these 2 simplices, and set
$$
\sigma_F^{add} = \sigma_F + \sigma_{F;[-+]} + \sigma_{F;\Delta^2}.
$$
Then by construction, $\sigma_F^{add}$ forms a mod-2 cycle.

This finishes the description of the
basic Lagrangian cycle. A similar description can be given in the higher dimensional cases,
which we will study elsewhere. This enables us to define the
following important Lagrangian cycle.

\begin{defn}[Basic Lagrangian cycle and cliff-wall surgery]\label{defn:selectorcycle}
Let $\dim N = 2$. We call the cycle $\sigma_{F}^{add}$ the basic Lagrangian cycle of
$\phi_F^1(o_N)$ (associated to the basic Lagrangian selector $\sigma_{F}$).
We call the replacement of $\phi_F^1(o_N)$ by the $\Sigma_F^{add}$ the
\emph{cliff-wall surgery} of the $\phi_F^1(o_N)$.
\end{defn}

\begin{rem}
\begin{enumerate}
\item
We also refer to \cite{kasturi-oh1,kasturi,kasturi-oh2}
for a usage of the general conormal variety of an open-closed domain with boundary and corners, which
also naturally occurs in micro-local analysis and in stratified Morse theory \cite{kash-schapira}.
\item
The basic Lagrangian cycle seems to be a good
replacement of non-graph type Lagrangian submanifold $\phi_F^1(o_N)$ in general
for the study of various questions arising in Hamiltonian dynamics and
symplectic topology. We hope to elaborate this point elsewhere.
\end{enumerate}
\end{rem}

\begin{rem} We believe that this surgery will play an important role in
the study of homotopy invariance of spectral invariants for the topological Hamiltonian paths
\cite{oh:hameo2}, which we hope to address elsewhere.
\end{rem}

\section{Lagrangian Floer homology and spectral invariants}
\label{sec:Lag-spectral}

In this section, we first briefly recall the construction of Lagrangian spectral invariants
$\rho^{lag}(H;a)$ for $L_H = \phi_H^1(o_N)$ performed by the author in \cite{oh:cag}.
A priori, this invariant may depend on $H$, not just on $L_H$ itself.
In \cite{oh:cag}, we prove that
\be\label{eq:rhoH=rhoK}
\rho^{lag}(H;a) = \rho^{lag}(F;a)
\ee
for all $a \in H^*(N;\Z)$ if $L_H = L_F$, \emph{but modulo the addition of a constant}
and then somewhat ad-hoc normalization to remove this ambiguity of a constant.

\subsection{Definition of Lagrangian spectral invariants}

Consider the zero section $o_N$ and the space
$$
\CP(o_N,o_N) = \{\gamma:[0,1] \to T^*N \mid \gamma(0), \, \gamma(1) \in o_N \}.
$$
The set of generators of $CF(H;o_N,o_N)$ is that of solutions
$$
\dot z = X_H(t,z(t)), \, z(0), \, z(1) \in o_N
$$
and its Floer differential is defined by counting the number of solutions of
\be\label{eq:CRHJ}
\begin{cases}
\dudtau + J\left(\dudt - X_H(u)\right) = 0\\
u(\tau,0), \, u(\tau,1) \in o_N.
\end{cases}
\ee
An element $\alpha \in CF(H;o_N,o_N)$ is expressed as a finite sum
$$
\alpha = \sum_{z \in \CC hord(H;o_N,o_N)} a_z [z], \quad a_z \in \Z.
$$
We define the level of the chain $\alpha$ by
\be\label{eq:lambdaH-lag}
\lambda_H(\alpha): = \max_{z \in \supp \alpha} \{\CA^{cl}_H(z)\}.
\ee
For given non-zero cohomology class $a \in H^*(N,\Z)$, we consider
its Poincar\'e dual $[a]^\flat:= PD(a) \in H_*(N,\Z)$  and
its image under the canonical isomorphism
$$
\Phi: H_*(N,\Z) \to HF_*(H,J;o_N,o_N).
$$
\begin{defn} Let $(H,J)$ be a Floer regular pair relative to
$(o_N,o_N)$ and let $(CF(H),\del_{(H,J)})$ be its associated Floer complex.
For any $0 \neq a \in H^*(N,\Z)$, we define
\be\label{eq:rhoHa}
\rho^{lag}(H;a) = \inf_{\alpha \in \Phi(a^\flat)}\{\lambda_H(\alpha)\}.
\ee
\end{defn}

One important result is the following basic property, called \emph{spectrality} in \cite{oh:alan},
which is not explicitly stated in \cite{oh:jdg} but can be easily derived by a compactness
argument. (See the proof in \cite{oh:alan} given in the Hamiltonian context.)

\begin{prop}\label{prop:tight} Let $H=H(t,x)$ be any, not necessarily nondegenerate,
smooth Hamiltonian. Then for any $0 \neq a \in H^*(N,\Z)$, there exists a
point $x \in L_H \cap o_N$ such that
$$
\CA^{cl}_H(z_x^H) = \rho^{lag}(H;a).
$$
In particular, $\rho^{lag}(H;a) \in \Spec(H;N)$.
\end{prop}

\subsection{Comparison of two Cauchy-Riemann equations}
\label{sec:comparison}

So far we have looked at the Hamiltonian-perturbed Cauchy-Riemann equation \eqref{eq:CRHJ},
which we call the {dynamical version
as in \cite{oh:jdg}.

On the other hand, one can also consider the genuine Cauchy-Riemann equation
\be\label{eq:dvtildeJ}
\begin{cases}\dvdtau + J^H \dvdt = 0 \\
v(\tau ,0) \in \phi_H^1(o_N), \,
v(\tau ,1) \in o_N
\end{cases}
\ee
for the path $u:\R \to \CP(o_N,L)$ where $L = \phi_H^1(o_N)$ and
$$
\CP(o_N,L) = \{\gamma: [0,1] \to T^*N \mid \gamma (0) \in L, \,  \gamma(1) \in o_N\}
$$
and $J^H_t = (\phi_H^t\phi_H^{-1})_*J_t$. We call this version the geometric version.

We now describe the geometric version of the Floer homology in some more details.
We refer readers to \cite{oh:jdg} for the discussion on the further comparison of the two
versions in the point of moduli spaces and others. The upshot is that there is
a filtration preserving isomorphisms between the dynamical version and
the geometric version of the Lagrangian Floer theories.

We denote by $\widetilde \CM(L_H,o_N;J^H)$ the set
of finite energy solutions and $\CM(L_H,o_N;J^H)$ to be its quotient by $\R$-translations.
This gives rise to the geometric version of the Floer homology
$HF_*(o_N, \phi_H(o_N), \widetilde J)$
of the type \cite{floer:Morse,oh:cag} whose generators are the intersection
points of $o_N\cap \phi_H (o_N)$.
An advantage of this version is that it depends only
on the Lagrangian submanifold $L = \phi_H(o_N)$, only loosely on $H$.
(The author proved in \cite{oh:cag} that $\rho (H;a)$ is the invariant of
$L_H = \phi_H(o_N)$ up to this normalization by comparing
these two versions of the Floer theory in \cite{oh:jdg,oh:cag}.)

The following is a straightforward to check but is a crucial lemma.

\begin{lem}\label{lem:equiv} Let $L = \phi_H^1(o_S)$.
\begin{enumerate}
\item The map $\Phi_H: o_N \cap L \to \CC hord(H;o_N,o_N)$ defined by
$$
x \mapsto z_x^H(t)= \phi_H^t\left((\phi_H^1)^{-1}(x)\right)
$$
gives rise to the one-one correspondence between the set
$o_N \cap L \subset \CP(o_N,L)$ as constant paths and the set of
solutions of Hamilton's equation of $H$.
\item The map $a \mapsto \Phi_H(a)$ also defines a one-one
correspondence from the set of solutions of \eqref{eq:CRHJ} and that of
\be\label{eq:dvJH}
\begin{cases}
\dvdtau + J^H \dvdt = 0 \\
v(\tau ,0) \in \phi_H(o_N),  \,  v(\tau ,1) \in o_N
\end{cases}
\ee
where $J^H = \{ J^H_t \} , J^H_t: = (\phi^t_H (\phi^1_H)^{-1})^* J_t $.
Furthermore, \eqref{eq:dvJH} is regular if and only if \eqref{eq:CRHJ} is regular.
\end{enumerate}
\end{lem}

Once we have transformed \eqref{eq:CRHJ} to \eqref{eq:dvJH}, we can further
deform $J^H$ to the constant family $J_0$ and consider
\be\label{eq:dvJ0}
\begin{cases}
\dvdtau + J_0 \dvdt = 0 \\
v(\tau ,0) \in \phi_H(o_N),  \, v(\tau ,1) \in o_N.
\end{cases}
\ee
This latter deformation preserves the filtration of
the associated Floer complexes \cite{oh:jdg}. A big advantage of
considering this equation is that it enables us to study the behavior of
spectral invariants for a sequence of $L_i$ converging to $o_N$
\emph{in weak Hamiltonian topology}.

The following proposition provides the action functional associated to
the equation \eqref{eq:dvJH}, \eqref{eq:dvJ0}, which will give a
natural filtration associated Floer homology $HF(L,o_N)$.

\begin{prop}\label{prop:effective} Let $L$ and $h_L$ be as in Lemma \ref{lem:generating}.
Let $\Omega(L,o_N;T^*N)$ be the space of paths $\gamma:[0,1] \to \R$
satisfying $\gamma(0) \in L, o_N, \, \gamma(1) \in o_N$. Consider the
effective action functional
$$
\CA^{\text{\rm eff}}(\gamma) = \int \gamma^*\theta + h_H(\gamma(0)).
$$
Then $d\CA^{\text{\rm eff}}(\gamma)(\xi) = \int_0^1
\omega(\xi(t),\dot\gamma(t)) \, dt$. In particular,
\be\label{eq:effCA=CA}
\CA^{\text{\rm eff}}(c_x) = h_H(x) = \CA^{cl}_H(z_x^H)
\ee
for the constant path $c_x \equiv x \in L\cap o_N$ i.e., for any
critical path $c_x$ of $\CA^{\text{\rm eff}}$.
\end{prop}

We would like to highlight the
presence of the `boundary contribution' $h_H(\gamma(0))$ in the definition of
the effective action functional above: This addition is needed to make the
Cauchy-Riemann equation \eqref{eq:dvtildeJ}
or \eqref{eq:dvJ0} into a \emph{gradient trajectory equation} of
the relevant action functional. We refer readers to section 2.4 \cite{oh:jdg}
and Definition 3.1 \cite{kasturi-oh1} and the discussion
around it for the upshot of considering the effective action functional and
its role in the study of Cauchy-Riemann equation.

\subsection{Triangle inequality for Lagrangian spectral invariants}
\label{subsec:triangle}

We recall from, \cite{schwarz},
\cite{oh:alan} that the triangle inequality of the Hamiltonian spectral invariants
$$
\rho^{ham}(H \# F;a\cdot b) \leq \rho^{ham}(H;a) + \rho^{ham}(F;b)
$$
for the product Hamiltonian $H \# F$ relies on the homotopy invariance
property of spectral invariants which in turn relies on the existence of
canonical normalization procedure of Hamiltonians on closed $(M,\omega)$ which is
nothing but the mean normalization. On the other hand, one can directly prove
$$
\rho^{ham}(H * F;a \cdot b) \leq \rho^{ham}(H;a) + \rho^{ham}(F;b)
$$
more easily for the concatenated Hamiltonian. (See e.g., \cite{fooo:bulk} for
the proof.) Once we have the latter inequality, we can derive the former from the
latter again by the homotopy invariance property of $\rho^{ham}(\cdot;a)$
\emph{for the mean-normalized Hamiltonians}.

When one attempts to assign an invariant of Lagrangian submanifold $\phi_H^1(o_N)$
itself out of the spectral invariant $\rho^{lag}(H;a)$, one has to choose a
normalization of the Hamiltonian relative to the Lagrangian submanifold.
Since there is no canonical normalization unlike the Hamiltonian case, the
invariance property of Lagrangian spectral invariants and so
the triangle inequality is somewhat more nontrivial than the case of
Hamiltonian spectral invariants. In this subsection, we clarify these issues of
invariance property and of the triangle inequality.

The following parametrization independence follows immediately from the
construction of Lagrangian spectral invariants and $L^{(1,\infty)}$-continuity of
$H \mapsto \rho^{lag}(H;a)$.

\begin{lem}\label{lem:Hchi}
Let $H=H(t,x)$ be any, not necessarily nondegenerate,
smooth Hamiltonian and let $\chi:[0,1] \to [0,1]$ a reparameterization
function with $\chi(0) = 0$ and $\chi(1) = 1$. Then
$$
\rho^{lag}(H;a) = \rho^{lag}(H^\chi;a)
$$
where $H^\chi(t,x) = \chi'(t)H(\chi(t),x)$.
\end{lem}

We first recall the following triangle inequality which was essentially
proved in \cite{oh:cag}. (See Theorem 6.4 and Lemma 6.5 \cite{oh:cag}.
In \cite{oh:cag}, the cohomological
version of the Floer complex was considered and hence the opposite
inequality is stated. Other than this, the same proof can be applied here.)

\begin{prop}\label{prop:oh-triangle} Let $H, \, F \in \CP C^\infty_{asc}(T^*N;\R)$,
and assume $F$ is autonomous. Then we have
\be\label{eq:oh-triangle}
\rho^{lag}(H \# F;ab) \leq \rho^{lag}(H;a) + \rho^{lag}(F;b).
\ee
\end{prop}

Monzner, Vichery, and Zapolsky \cite{MVZ} proved the
following form of the triangle inequality which uses the concatenated Hamiltonian $H * F$
instead of the product Hamiltonian $H \# F$.

\begin{prop}[Proposition 2.4 \cite{MVZ}]\label{prop:Lag-triangle}
Let $H, \, K$ be compactly supported.
Suppose $H(1,x) \equiv F(0,x)$ and $H * F$ be the concatenated Hamiltonian.
Then
\be\label{eq:Lag-triangle}
\rho^{lag}(H * F;ab) \leq \rho^{lag}(H;a) + \rho^{lag}(F;b)
\ee
for all $a, \, b \in H^*(N)$.
\end{prop}

In particular, this proposition applies to all pairs $H, \, F$ which are
compactly supported and boundary flat.

\begin{rem} We suspect that \eqref{eq:oh-triangle} holds even for the non-autonomous $F$
as in the Hamiltonian case but we did not check this, since it is not needed in the present paper.
\end{rem}

\subsection{Assigning spectral invariants to Lagrangian submanifolds}
\label{subsec:assign}

In this subsection, we identify a class, denoted by $\CP C^\infty_{(B;e)}$, of Hamiltonians
$H$ among those satisfying $\phi_H^1(o_N) = \phi_F^1(o_N)$, such that the equality
$$
\rho^{lag}(H;a) = \rho^{lag}(F;a)
$$
holds for all $H, \, F \in \CP C^\infty_{ass;B}$. As the notation suggests,
the class depends on the subset $B \subset N$.

We start with the following proposition. The proof closely follows that of Lemma 2.6
\cite{MVZ} which uses Proposition \ref{prop:Lag-triangle} in a significant way. We
need to modify their proof to obtain a somewhat stronger statement,
which replaces the condition ``$\phi_H^1= \phi_F^1$'' used in \cite{MVZ} by the conditions
put in this proposition.

\begin{prop}[Compare with Lemma 2.6 \cite{MVZ}] \label{prop:rhoH=rhoK}
Let $H, \, F \in \CP C^\infty_{asc}(T^*N;\R)$ be boundary-flat. Suppose in addition $H, \, F$ satisfy the following:
\begin{enumerate}
\item $\phi_H^1(o_N) = \phi_F^1(o_N)$,
\item $H \equiv c(t)$, $F \equiv d(t)$ on a tubular neighborhood
$T \supset B$ in $T^*N$ of a closed ball $B \subset o_N$ where $c(t), \, d(t)$
are independent of $x \in T$, and
\item they satisfy
$$
\int_0^1 c(t) \, dt = \int_0^1 d(t)\, dt.
$$
\end{enumerate}
Then $\rho^{lag}(H;a) = \rho^{lag}(F;a)$ holds for all $a \in H^*(N,\Z)$ without ambiguity of constant.
\end{prop}
\begin{proof}

We consider the Hamiltonian path $\phi_G: t\mapsto \phi_G^t$ with
$G = \widetilde F * H$ with $\widetilde F(t,x) = - F(1-t,x)$.
This defines a loop of Lagrangian submanifold
$$
t \mapsto \phi_G^t(o_N), \quad \phi_G^1(o_N) = o_N
$$
and satisfies $\phi_G^t|_{B} \equiv id$ and
$$
G(t,q) = \begin{cases} - c(1-2t) \quad & 0 \leq t \leq 1/2\\
d(2t-1) \quad & 1/2 \leq t \leq 1
\end{cases}
$$
for all $q \in B \subset T$ by definition $G = \widetilde F * H$.

We claim $\rho^{lag}(G;a) = 0$ for all $0 \neq a \in H^*(N)$. This will be an immediate
consequence of the following lemma and the spectrality of numbers $\rho^{lag}(G;a)$.
\begin{lem}
The value $\CA^{cl}_G(z)$ does not depend on
the Hamiltonian chord $z \in \CC hord(G;o_N,o_N)$. In particular,
$\CA^{cl}_G(z) = 0$.
\end{lem}
\begin{proof} Recall that any Hamiltonian chord in $\CC hord(G;o_N,o_N)$
has the form
$$
z(t) = z_G^q(t)
$$
for some $q \in o_N$. Here we use the hypothesis $\phi_G^1(o_N) = o_N$. Consider any smooth
path $\alpha:[0,1] \to o_N$ with $\alpha(0) = q, \, \alpha(1) = q'$. Then
$$
\CA^{cl}_G\left(z_G^{q'}\right) - \CA^{cl}_G\left(z_G^q\right)
= \int_0^1 \frac{d}{du} \CA^{cl}_G\left(z_G^{\alpha(u)}\right) \, du.
$$
But a straightforward computation using the first variation formula \eqref{eq:1stvariation} implies
$$
\frac{d}{du} \CA^{cl}_G\left(z_G^{\alpha(u)}\right)
= \left\langle \theta, \frac{\del}{\del u}(\phi_G(\alpha(u)))\right\rangle
-\left\langle \theta, \frac{\del}{\del u}(\alpha(u))\right\rangle = 0 - 0 = 0
$$
since $\phi_G(\alpha(u)), \, \alpha(u) \in o_N$.

For the second statement, we have only to consider the constant path $z\equiv c_q \in B$
for which
\beastar
\CA^{cl}_G(c_q) & = & -\int_0^1 G(t,q)\, dt = \int_0^{1/2} c(1-2t)\, dt - \int_{1/2}^1 d(2t-1)\, dt\\
& = & \int_0^1 c(t)\,dt - \int_0^1 d(t)\,dt =0.
\eeastar
This proves the lemma.
\end{proof}

Once we have the lemma, we can apply the triangle inequality \eqref{eq:Lag-triangle}
$$
\rho^{lag}(H;a) \leq \rho^{lag}(F;a) + \rho^{lag}(G;1) = \rho^{lag}(F;a)
$$
for any given $a \in H^*(N)$.
By changing the role of $H$ and $F$ in the proof of the above lemma, we also obtain
$\rho^{lag}(\widetilde G;1) = 0$ and then obtain $\rho^{lag}(F;a) \leq \rho^{lag}(H;a)$ by triangle
inequality. This finishes the proof of the proposition.
\end{proof}

This proposition motivates us to introduce the following definitions

\begin{defn}\label{defn:PCBe} For each given $B \subset N$, we define
$$
\mathfrak{Iso}_B(o_N;T^*N) = \{ L \in \mathfrak{Iso}(o_N;T^*N) \mid
o_N \cap L \supset o_B\}.
$$
When a function $c:[0,1] \to \R$  is given in addition, we define
\beastar
\CP C^\infty_{(B;e)} & = & \{ H \in \CP C^\infty_{asc} \mid H_t \equiv c(t) \,
\mbox{on a neighborhood of $o_B$ in $T^*N$}\\
&{}& \hskip1in \mbox{and } \, \int _0^1 c(t)\, dt=e\}.
\eeastar
\end{defn}

With these definitions, the proposition enables us to unambiguously define
the following spectral invariant attached to $L$.

\begin{defn}\label{defn:rhoBe}
Suppose $L \in \mathfrak{Iso}_B(o_N;T^*N)$ and let $e \in \R$ be given.
For each given such $e$, we define a spectral invariant of
$L \in \mathfrak{Iso}_{(B;e)}(o_N;T^*N)$ by
$$
\rho^{(B;e)}(L;a): = \rho^{lag}(H;a), \quad L = \phi_H^1(o_N)
$$
for a (and so any) $H \in \CP C^\infty_{(B;e)}$.
\end{defn}

With this definition, we have the following obvious lemma

\begin{lem} Let $H \in \CP C^\infty_{(B;e)}$, then $\widetilde H, \, \overline H \in
\CP C^\infty_{(B;-e)}$.
\end{lem}

Then we prove the following duality statement of $\rho^{(B;e)}$.

\begin{prop}\label{prop:duality}
Let $H \in \CP C^\infty_{(B;e)}$ and $L = \phi_H^1(o_N)$.
We denote $\widetilde L = \phi_{\widetilde H}^1(o_N) = \phi_{\overline H}^1(o_N)$.
Then
\be\label{eq:duality}
\rho^{(B;-e)}(\widetilde L;1) = -\rho^{(B;e)}(L;[pt]^\#).
\ee
\end{prop}
\begin{proof} By the above lemma, $\widetilde H \in \CP C^\infty_{(B;-e)}$ and so
$\rho^{(B;-e)}(\widetilde L;1)$ is given by
$$
\rho^{(B;-e)}(\widetilde L;1) = \rho^{lag}(\widetilde H;1)
$$
by definition. But it was proven in \cite{viterbo:generating,oh:jdg,oh:cag} that
\be\label{eq:duality-H}
\rho^{lag}(\widetilde H;1) = - \rho^{lag}(H;[pt]^\#)
\ee
which follows from the Poincar\'e duality argument, by studying the
time-reversal flow of the Floer equation \eqref{eq:CRHJ} $\widetilde u$
defined by $\widetilde u(\tau,t) = u(-\tau, 1-t)$. The map $\widetilde u$ satisfies
the equation
$$
\begin{cases}
\frac{\del \widetilde u}{\del \tau} + \widetilde J\left(\frac{\del \widetilde u}{\del t}
- X_{\widetilde H}(\widetilde u)\right) = 0 \\
\widetilde u(\tau,0), \, \widetilde u(\tau,1) \in o_N.
\end{cases}
$$
Furthermore this equation is compatible with the involution of the path space
$$
\iota: \Omega(o_N,o_N) \to \Omega(o_N,o_N)
$$
defined by $\iota(\gamma)(t) = \widetilde \gamma(t)$ with $\widetilde \gamma(t) = \gamma(1-t)$
and the action functional identity
$$
\CA^{cl}_{\widetilde H}(\widetilde \gamma) = - \CA^{cl}_H(\gamma).
$$
We refer to \cite{oh:cag} for the details of the duality argument in the Floer
theory used in the derivation of \eqref{eq:duality-H}.

On the other hand, by definition,
$$
\rho^{lag}(H;[pt]^\#) = \rho^{(B;e)}(L;[pt]^\#)
$$
since $H \in \CP C^\infty_{(B;e)}$. This finishes the proof.
\end{proof}

\section{Comparison theorem of $f_H$ and $\rho^{lag}(H;1)$}
\label{sec:cap}

We first remark that both $\rho^{lag}(H;1)$ and $f_H$ remain unchanged
under the change of $H$ outside a neighborhood of $\bigcup_{t \in [0,1]} \phi_H^t(o_N)$.

The main theorem we prove in this section is the following
which is closely related to Proposition 5.1 \cite{viterbo:generating}.

\begin{thm}\label{thm:rhoversusfH}
For any Hamiltonian $H \in \CP C^\infty_{asc}$,
$$
\rho^{lag}(H;[pt]^\#) \leq \min f_H, \quad \max f_H \leq \rho^{lag}(H;1).
$$
\end{thm}

For the purpose of studying comparison result given in the next section, we start with this section by
adding the following additional symmetry property of $f_H$ and $\rho^{lag}$
under the reflection $\frak r: T^*N \to T^*N$ defined by $\frak r(q,p) = (q,-p)$.
Such a reflection argument was used by Viterbo \cite{viterbo:generating} in the
proof of similar identities in the context of generating function method.

\subsection{Anti-symplectic reflection and basic phase function}
\label{subsec:reflection}

\begin{prop}\label{prop:ftildeH=-fH} Consider the
canonical reflection map $\frak r: T^*N \to T^*N$ given by $\frak r(q,p) = (q,-p)$
and define the Hamiltonian $H^\frak r$ to be $H^\frak r(t,x) = - H(t, \frak r(x))$ for $x = (q,p)$.
Then
$$
f_{H^\frak r} = - f_H, \quad \rho^{lag}(H^\frak r;1) = - \rho^{lag}(H;[pt]^\#)
$$
\end{prop}
\begin{proof} We observe that the map satisfies $\frak r^*\theta = - \theta$ and
in particular is anti-symplectic. It also preserves the zero section and each individual
fibers of $T^*N$ and so induces the corresponding reflection map on the path space
$$
\frak r: \Omega(L, T_q^*N) \to \Omega(L^\frak r, T_q^*N); z = (q,p) \mapsto \frak r(z) = (q, -p)
$$
for each given base point $q \in N$,
where $L^\frak r := \frak r(L) = \phi_{H^\frak r}^1(o_N)$. A straightforward
computation also shows
\be\label{eq:AHrz}
\CA_{H^\frak r}(\frak r(z)) = - \CA_H(z).
\ee
We then consider $J$' satisfying $\frak r^*J = - J$.
For example, the standard Sasakian almost complex structure $J_g$
associated any Riemmanian metric $g$ on $N$ \cite{floer:witten} is such an almost complex
struture. Therefore the set of such
$J$'s is non-empty. It is also not difficult show that the set is
a contractible infinite dimensional manifold. (See Lemma 4.1 \cite{fooo:antisymp} for its proof.)

Then a straightforward computation shows that
this reflection map induces one-one correspondence
$$
u \mapsto u'; \quad u'(\tau,t): =  \frak r(u(-\tau,t))
$$
between the set of solutions of the Floer equation \eqref{eq:CRHJqN} and those associated to
$$
\begin{cases}
\frac{\del u'}{\del \tau} + J \left(\frac{\del u'}{\del t}
- X_{H^{\frak r}}(u')\right) = 0 \\
u'(\tau,0) \in o_N, \, u'(\tau,1) \in T_q^*N.
\end{cases}
$$
Furthermore all the generic transversality statements are equivalent for $u$ and $u'$
for $J$'s satisfying $\frak r^*J = - J$ via the transformation of the Hamiltonian $H \mapsto H^\frak r$.
Therefore $\frak r$ induces canonical isomorphism
$$
\frak r_*: HF_*(H;o_N,T_q^*N) \to HF_*(H^\frak r;o_N, T_q^*N).
$$
We also recall the canonical isomorphism established for arbitrary generic $H$ in \cite{oh:jdg}
$$
HF_*(H;o_N,T_q^*N) \cong H_*(\{pt\}) \cong \Z
$$
which has rank 1. Therefore $(\frak r)_*([pt]_H) = \pm [pt]_{H^\frak r}$.
The first equality then follows from these observations and \eqref{eq:AHrz}
by the general construction of spectral invariants
$\rho^{lag}(H;\{q\})$ given in section \ref{sec:basic},
especially (the Lagrangian version of) Conformality Axiom \cite{oh:alan}.

A similar consideration based on \eqref{eq:AHrz} with the boundary condition
$$
u'(\tau,0) \in o_N, \, u'(\tau,1) \in o_N
$$
gives rise to the second identity by the same kind of duality argument as done to prove \eqref{eq:duality-H}
in \cite{oh:cag}. We omit the details by referring readers thereto for the details. This finishes the proof.
\end{proof}

\subsection{Analysis of Example 9.4 \cite{oh:jdg}}
\label{subsec:example}

Before giving the proof of Theorem \ref{thm:rhoversusfH}, we illustrate the inequalities by
a concrete example, which is a continuation of Example 9.4 \cite{oh:jdg}.

\begin{exm}
Consider the Lagrangian submanifold $L$ in $T^*S^1$ pictured as in Figure 1
whose coordinates we denote by $(q,p)$.
One can check that the wave front projection of $L$, i.e., the graph of the multi-valued function
$h_H$ of the associated Hamiltonian $H$ such that $L = \phi_H^1(o_{S^1})$
can be drawn as in Figure 2 in $S^1 \times \R$ whose coordinates we denote by $(q,a)$.

Here we denote by $z_i = (q_i,0)$ below for $i = 0, \cdots, 3$ the intersections of $L$ with the zero
section, and by $x_i$ $i=1,\, 2$ the caustics and by $y$ the point at which the two
regions between the graph and the dotted line have the same area
in Figure 1. Note that the points $z_i$'s are the critical points of the
multi-valued generating function $h_H$
(or correspond to critical points of the action functional),
$x_i$'s to the cusp points of the wave front and $y$ is the
crossing point of two different branches of the wave front projection.

Using the continuity of the basic phase function $f_H$ where $L = \phi_H^1(o_N)$, one can easily see
that the graph of $f_H$ is the one bold-lined in Figure 2.  We would like to
note that the value $\min_{q\in N} f_H(q)$ is {\it not} a critical value of
$\CA_H^{cl}$, and the branch of the wave front containing the point $(q_1,a_1)$
associated to the critical point $z_1$ of $h_H$ is eliminated
from the graph of the basic phase function $f_H$.

\begin{figure}[tbp]
\centering
\includegraphics{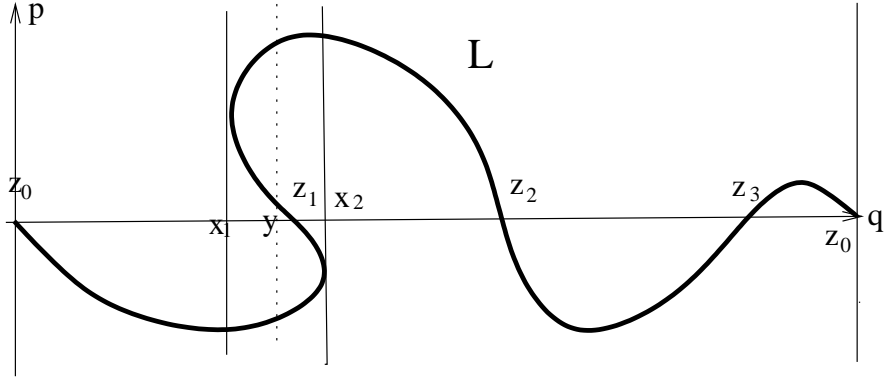}
\caption{Multi-section Lagrangian $L$}
\end{figure}

We note that the Floer complex $CF(L,o_{S^1}) \cong \oplus_{i=0}^3 \Z\{z_i\}$
and its boundary map is given by
$$
\del(z_0) = z_1 - z_3, \, \del(z_1) = 0 = \del(z_3), \, \del(z_2) = z_3 - z_1.
$$
(Here we take $\Z_2$-coefficients to avoid precise checking of the signs which is
irrelevant for the study of this example.)
From this we derive
$$
\ker \del = \Z \{z_1,z_3, z_0 + z_2\}, \quad \Im \del = \Z \{z_1 - z_3\}.
$$
Therefore the class $1$ is realized by the Floer cycle $z_0 + z_2$ (or any other class of the
form $z_0 + z_2 + \del(\alpha)$)
and the class $[pt]^\#$ is realized by the Floer cycle of the form $z_1$ (or any other class of
the form $z_1 + \del(\beta)$. A simple examination of Figure 1 and 2 and comparison of the action values
associated to the intersection points $z_1$ and $z_3$
shows that the infimum of the level $\lambda_H(z_1 + \del(\beta))$, which is nothing by $\rho(H;[pt]^\#)$ by definition,
is realized by the Floer cycle represented by the intersection point $z_1$. Therefore we obtain
$$
\rho(H;[pt]^\#) = \CA_{H}(z_1^H)
$$
which is denoted as $a_1$ in Figure 2, where $z_1^H$ is the Hamiltonian path given by $z_1^H(t) := \phi_H^t(\phi_H^1)^{-1}(z_1)$.

Combining the above discussion on $\rho^{lag}$ and comparing them with the values of $f_H$, we can easily obtain from Figure 2 that
$$
\rho(H;[pt]^\#) < \min f_H < \max f_H = \rho(H;1).
$$
It is interesting to observe two peculiar phenomena in this example:
\begin{enumerate}
\item the minimum of $f_H$ is realized at a non-smooth point $y \in N$
of the function $f_H$, and
\item the value $\rho(H;[pt]^\#)$ is realized by the
`local maximum' of the branch of $h_H$ containing the point $(q_1, a_1) \in S^1 \times \R$ where $q_1 = \pi(z_1)$
and $a_1 = \CA_{H}(z_1^H)$.
\end{enumerate}

\begin{figure}[tbp]
\centering
\includegraphics{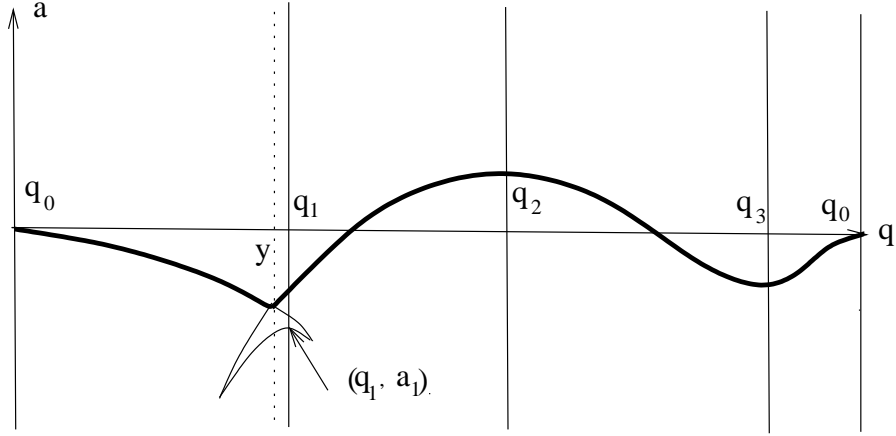}
\caption{Wave front of $L$ and the graph of $f_H$}
\end{figure}

\end{exm}

\subsection{Proof of comparison result on $\rho^{lag}(H;1)$ and $f_H$}
\label{subsec:proof}

We now go back to the proof of Theorem \ref{thm:rhoversusfH}.
We first remark that the second inequality in Theorem \ref{thm:rhoversusfH}
immediately follows by applying the first inequality to the Hamiltonian $H^\frak r$
and combining Proposition \ref{prop:ftildeH=-fH}.

Therefore it remains to prove the inequality $\max f_H \leq \rho(H;1)$.
which will occupy the rest of this section.

We first recall the definition of the triangle product described in \cite{oh:cag}, \cite{foh:ajm}
and put it into a more modern context in the general Lagrangian Floer theory
such as in \cite{fooo:book} and in other more recent literatures.

Let $q \in N$ be given. Consider the Hamiltonians $H: [0,1] \times T^*N \to \R$
such that $L_H$ intersects transversely both $o_N$ and $T_q^*N$.
We consider the Floer complexes
$$
CF(L_H,o_N), \quad CF(o_N,T_q^*N), \quad CF(L_H,T_q^*N)
$$
each of which carries filtration induced from the effective action function
given in Proposition \ref{prop:effective}. We denote by $\mathfrak v(\alpha)$ the level
of the chain $\alpha$ in any of these complexes.

More precisely, $CF(L_H,o_N)$ is filtered by the effective functional
$$
\CA^{(1)}(\gamma):= \int \gamma^*\theta + h_H(\gamma(0)),
$$
$CF^{\mu}(o_N,T_q^*N)$ by
$$
\CA^{(2)}(\gamma) := \int \gamma^*\theta,
$$
and $CF(L_H,T_q^*N)$ by
$$
\CA^{(0)}(\gamma):= \int \gamma^*\theta + h_H(\gamma(0))
$$
respectively. We recall the readers that $h_H$ is the potential of $L_H$
and the zero function the potentials of $o_N, \, T_q^*N$.

We now consider the triangle product in the chain level, which we denote by
\be\label{eq:m2chain}
\mathfrak m_2: CF(L_H,o_N) \otimes CF(o_N,T_q^*N) \to CF(L_H,T_q^*N)
\ee
following the general notation from \cite{fooo:book}, \cite{seidel:triangle}.
This product is defined by considering all triples
$$
x_1 \in L_H \cap o_N, \, x_2 \in o_N \cap T_q^*N, \, x_0 \in L_H \cap T_q^*N
$$
with the polygonal Maslov index $\mu(x_1,x_2;x_0)$ whose
associated analytical index, or the virtual dimension of the moduli space
$$
\CM_{3}(D^2; x_1,x_2; x_0): = \widetilde \CM_3(D^2;x_1,x_2;x_0)/PSL(2,\R)
$$
of $J$-holomorphic triangles, becomes zero and counting the number of elements thereof.
The precise formula of the index is irrelevant to our discussion which, however, can be
found in \cite{seidel:triangle}, \cite{fooo:anchor}.

\begin{defn}\label{defn:CM3} Let $J = J(z)$ be a domain-dependent family of
compatible almost complex structures with $z \in D^2$.
We define the space $\widetilde \CM_3(D^2;x_1,x_2;x_0)$
by the pairs $(w,(z_0,z_1,z_2))$ that satisfy the following:
\begin{enumerate}
\item $w: D^2 \to T^*N$ is a continuous map satisfying $\delbar_J w = 0$
$D^2 \setminus \{z_0,z_1,z_2\}$,
\item the marked points $\{z_0,z_1,z_2\} \subset \del D^2$ with counter-clockwise
cyclic order,
\item $w(z_1) = x_1, \, w(z_2)=x_2$ and $w(z_0) = x_0$,
\item the map $w$ satisfies the Lagrangian boundary condition
$$
w(\del_1 D^2) \subset L_H, \, w(\del_2 D^2) \subset o_N, \, w(\del_3 D^2) \subset T_q^*N
$$
where $\del_i D^2 \subset \del D^2$ is the are segment in between $x_i$ and $x_{i+1}$
($i \mod 3$).
\end{enumerate}
\end{defn}
The general construction is by now well-known and e.g., given in \cite{fooo:book}. In the
current context of exact Lagrangian submanifolds, the detailed construction is also given
in \cite{oh:cag} and \cite{seidel:triangle}. One important ingredient in relation to
the study of the effect on the level of Floer chains under the product is the following
(topological) energy identity where the choice of the \emph{effective}
action functional plays a crucial role. For readers' convenience, we give its
proof here.

\begin{prop} Suppose $w:D^2 \to T^*N$ be any smooth map with finite energy
that satisfy all the conditions given in \ref{defn:CM3}, but not necessarily $J$-holomorphic.
We denote by $c_x:[0,1] \to T^*N$ the constant path with its value $x \in T^*N$.
Then we have
\be\label{eq:area}
\int w^*\omega_0 = \CA^{(1)}(c_{x_1}) + \CA^{(2)}(c_{x_2}) - \CA^{(0)}(c_{x_0}).
\ee
\end{prop}
\begin{proof} Recall $\omega_0 = -d\theta$ and $i^*\theta = dh_H$ on $L_H$ and $i^*\theta = 0$ on
$o_N$ and $T_q^*N$ where $i$'s are the associated inclusion maps of $L_H, \, o_N, \, T_q^*N \subset T^*N$
respectively. Therefore
\beastar
\int_{D^2} w^*\omega_0 & = & - \int_{\del D^2} w^*\theta
= - \int_{\del_1 D^2} w^*\theta - \int_{\del_2 D^2} w^*\theta - \int_{\del_3 D^2} w^*\theta \\
& = & - \int_{\del_1 D^2} w^*dh_H - 0 - 0 = h_H(w(z_1)) - h_H(w(z_2)) \\
& = & h_H(x_1) - h_H(x_0) = \CA^{(1)}(c_{x_1}) - \CA^{(0)}(c_{x_0}) \\
& = & \CA^{(1)}(c_{x_1}) + \CA^{(2)}(c_{x_2}) - \CA^{(0)}(c_{x_0}).
\eeastar
Here the last equality comes since $\CA^{(2)}(c_{x_2}) = \int c_{x_2}^*\theta = 0$.
This finishes the proof.
\end{proof}

An immediate corollary of this proposition from the definition of $\mathfrak m_2$ is that
the map \eqref{eq:m2chain} restricts to
$$
\mathfrak m_2: CF^{\lambda}(L_H,o_N) \otimes CF^{\mu}(o_N,T_q^*N) \to CF^{\lambda + \mu}(L_H,T_q^*N).
$$
It is straightforward to check that this map satisfies
$$
\del(\mathfrak m_{2}(x,y)) = \mathfrak m_2(\del(x),y) \pm \mathfrak m_2(x,\del(y)
$$
and in turn induces the product map
\be\label{eq:m2hom}
*_F: HF^{\lambda}(L_H,o_N) \otimes HF^{\mu}(o_N,T_q^*N) \to HF^{\lambda+\mu}(L_H,T_q^*N)
\ee
in homology. This is because if $w$ is $J$-holomorphic $\int w^*\omega \geq 0$.
(We refer to \cite{oh:cag} and \cite{foh:ajm} for the general construction of
product map $\mathfrak m_2$ and to \cite{oh:cag}, \cite{MVZ} for the study of filtration.
Similar study of filtration is also performed in \cite{schwarz}, \cite{oh:alan} in the
Hamiltonian Floer homology setting.)

With these preparations, we are ready to wrap-up the proof of Theorem \ref{thm:rhoversusfH}:

\begin{proof}[Proof of Theorem \ref{thm:rhoversusfH}]
We consider a Floer cycle $\alpha$ representing the fundamental class
$1^\flat = [M] \in HF(L_H,o_N)$ and $\beta=\{q\}$ representing the unique generator of
$HF(o_N,T_q^*N) \cong \Z$. Then by definition
$$
\mathfrak v(\alpha) \geq \rho^{lag}(H;1), \quad \mathfrak v(\beta) = \rho^{lag}(0;[q]) = 0.
$$
Then its product cycle $\mathfrak m_2(\alpha,\beta) \in CF(L_H,T_q^*N)$
represents the homology class $[q] \in CF(L_H,T_q^*N) \cong \Z$ and so
$ \mathfrak v(\mathfrak m_2(\alpha,\beta)) \geq \rho^{lag}(H;\{q\})= f_H(q)$ by definition of the latter.
Applying the triangle inequality, we obtain
$$
\mathfrak v(\alpha) + 0 = \mathfrak v(\alpha) + \mathfrak v(\beta) \geq \mathfrak v(\mathfrak m_2(\alpha,\beta))
\geq \rho^{lag}(H;\{q\}) = f_H(q).
$$
Therefore we have derived
$$
\mathfrak v(\alpha) \geq f_H(q)
$$
for all cycle $\alpha \in CF(L_H,o_N)$ representing $[M]$. By definition of $\rho^{lag}(H;1)$, this proves
$$
\rho^{lag}(H;1) \geq f_H(q).
$$
Since this holds for any point $q \in N$, we have proved $\rho^{lag}(H;1) \geq \max f_H$.
\end{proof}

\section{A Hamiltonian $C^0$ continuity of spectral Lagrangian capacity}
\label{sec:Hamiltonian}

We first recall the definition of the function
$\gamma^{lag}_B: \frak{Iso}_B(o_N;T^*N) \to \R$
defined by
$$
\gamma^{lag}_B(L) = \rho^{lag}_B(H;1) - \rho^{lag}_B(L_H;[pt]^\#)
$$
for $L = \phi_H^1(o_N)$ with $H \in \CP^\infty_{ass;B}$.

In this section, we prove the following Hamiltonian $C^0$-continuity result of the function
which is the Lagrangian analog to Theorem 1 \cite{seyfad}.

\begin{thm}\label{thm:rhoC0conti} The function $\gamma^{lag}_B: \frak{Iso}_B(o_N;T^*N) \to \R$
is continuous with respect to the Hamiltonian $C^0$-topology in the sense of Definition \ref{defn:hamC0dist}.
\end{thm}

The triangle inequality of $\gamma^{lag}$
stated in section \ref{subsec:triangle} implies the inequalities
\beastar
\left|\gamma^{lag}_B(L_1) - \gamma^{lag}_B(L_2)\right| \leq
\max\left\{\gamma^{lag}_B\left(\phi_{H^2}^{-1}(\phi_{H^1}(o_N));o_N)\right),
\gamma^{lag}_B\left(\phi_{H^1}^{-1}(\phi_{H^2}(o_N));o_N)\right)\right\}.
\eeastar
We also note that for $L_k = \phi_{H^k}^1(o_N) \in \frak{Iso}_B(o_N;T^*N)$ for $k = 1, \, 2$
$$
\max\left\{d_{C^0}(\phi_{H^2}^{-1}(\phi_{H^1}(o_N));o_N)), d_{C^0}(\phi_{H^1}^{-1}(\phi_{H^2}(o_N));o_N))\right\} \to 0
$$
if and only if
$$
\max\left\{
d_{C^0}((\phi_{H^1}(o_N)),\phi_{H^2}(o_N)), d_{C^0}(\phi_{H^1}^{-1}(o_N));(\phi_{H^2})^{-1}o_N))\right\}
\to 0
$$
provided we assume $\supp \phi_{H^k}$ is compact and so
$\supp \phi_{H^k} \subset D^R(T^*N) \setminus T$, $k = 1,\, 2$,
for some $R > 0$ and $T \supset B$. The latter assumption is already embedded in the definition of
Hamiltonian topology given in Definition \ref{defn:hamC0dist}.

Therefore to prove the above theorem, it
is enough to prove the continuity of $\gamma^{lag}_B$
at the zero section $o_N$ in $\frak{Iso}_B(o_N;T^*N)$.

By unravelling the definition of Hamiltonian $C^0$-topology on $ \frak{Iso}_B(o_N;T^*N) $
given in  Definition \ref{defn:hamC0dist}, we now rephrase the continuity statement
at the zero section $o_N$ more explicitly. For this purpose, we introduce the notation
$$
\osc_{C^0}(\phi_H^1;o_N): =
\max\left\{\max_{x \in o_N}d\left(\phi_H^1(x),x\right), \,
\max_{x \in o_N}d\left((\phi_H^1)^{-1}(x),x\right)\right\}.
$$
Then it is easy to see that this continuity at $o_N$ is equivalent to the following

\begin{thm}\label{thm:rhoC0conti}
Let $\lambda_i = \phi_{H_i}$ where $H_i \in \CP C^\infty_{asc}$ is a sequence such that
\begin{enumerate}
\item $H_i \in \CP C^\infty_{R,K}$ for some $R, \, K > 0$ for all $i$ and $s \in [0,1]$,
\item There exists a closed ball $B \subset N$ such that
$\phi_{H_i}^t \equiv id$ on $B$ for all $t \in [0,1]$ for all $i$.
\item There exists a uniform neighborhood $T \supset o_B$ in $T^*N$ such that $\phi_{H_i}^1 \equiv id$
on $T$ for all $i$.
\item  $\lim_{i \to \infty}\osc_{C^0}(\phi_{H_i}^1;o_N)= 0$.
\end{enumerate}
Then
$$
\lim_{i \to \infty}\left(\rho^{lag}(H_i;1) - \rho^{lag}(L_{H_i};[pt]^\#)\right) = 0.
$$
\end{thm}

The proof of this theorem is an adaptation to the
Lagrangian context of the one used by Seyfaddini in his proof of
Theorem 1 (or rather Corollary 1.3) \cite{seyfad}. The proof is also
a variation of Ostrover's scheme used in \cite{ostrov} and is an
adaptation thereof. In our proof, we however use the Lagrangian analog to the notion of
`$\e$-shiftability' introduced by Seyfaddini \cite{seyfad}, instead
of `displaceability' used in \cite{ostrov} and in other literature
such as \cite{entov-pol:morphism}, \cite{usher:sharp}. In the Lagrangian context here,
the $\e$-shiftable domain is realized
as the graph of $df$ of a function $f$ having no critical points on the corresponding domain.
In this regard, it appears to the author
that the notion of $\e$-shiftability becomes more geometric and intuitive in the
Lagrangian context than in the Hamiltonian context.

\subsection{$\e$-shifting of the zero section by the differential of function}
\label{subsec:shifting}

Fix a Riemannian metric $g$ and the Levi-Civita connection on $N$. They
naturally induces a metric on $T^*N$. Denote the latter metric on $T^*N$ by
$\widetilde g$ and the corresponding distance function by
$\widetilde d(x,y)$ for $x, \, y \in T^*N$.
We denote by $D^r(T^*N)$ the disc bundle of $T^*N$ of radius $r$.

The following is the well-known
fact on this metric $\widetilde g$, which can be easily checked.

\begin{lem}\label{lem:rNg} The metric $\widetilde g$ carries following properties:
\begin{enumerate}
\item $\widetilde g$ is invariant under the reflection $(q,p) \mapsto (q,-p)$ and
in particular $o_N$ is totally geodesic.
\item There exists a sufficiently small $r = r(N,g) > 0$ depending only on $(N,g)$
such that
\begin{enumerate}
\item for all $d(q,q')< r$ $\widetilde d(o_q,o_{q'}) = d(q,q')$,
\item for all $x \in D^r(T^*N)$, which we denote $x = (q(x),p(x))$,
\be\label{eq:g-triangle}
d(o_{q(x)},x) \geq \max\{|p(x)|, d(q,q(x))\} \geq |p(x)|
\ee
where $|p(x)|$ is the norm on $T_{q(x)}^*N$.
\end{enumerate}
\end{enumerate}
\end{lem}

From now on, we will drop `tilde' from $\widetilde d$ and just denote by $d$
even for the distance function of $\widetilde g$ on $T^*N$ which should not confuse readers.

Consider the subset
$$
C^\infty_{crit}(N;B) = \{f \in C^\infty(N) \mid \Crit f \subset \operatorname{Int} B\}.
$$
The set $C^\infty_{crit}(N;B) \subset C^\infty(N)$ has the filtration
$$
C^\infty_{crit}(N;B) = \bigcup_T C^\infty_{crit}(N;B,T)
$$
where $C^\infty_{crit}(N;B,T)$ is the subset of $C^\infty_{crit}(N;B)$
that consists of $f$'s satisfying
\be\label{eq:dfBT}
\Graph (df|_B) \subset T.
\ee
It is easy to check that $C^\infty_{crit}(N;B,T) \neq \emptyset$ for any such $T \supset o_B$
by considering the $\lambda f$ for a sufficiently small $\lambda > 0$ for any given
Morse function $f$ with $\Crit f \subset \Int B$.

We now introduce the collection, denoted by $\CT_{(B;r)}$,
of the pairs $(T,f)$ consisting of a tubular neighborhood $T \supset o_B$
in $T^*N$ and a Morse function $f \in C^\infty_{crit}(N;B,T)$ such that
\be\label{eq:dfDr}
\Graph df \subset D^r(T^*N)
\ee
for the constant $r = r(N,g)$ given in Lemma \ref{lem:rNg}.

By the choice of the pair $(T,f) \in \CT_{(B;r)}$, we have
$$
\min\left\{\min_{p \in N \setminus B} |df(p)|, \,
d_{\text{\rm H}}(N \setminus B, \Crit f)\right\} > 0.
$$
where $d_{\text{H}}(N\setminus B,\Crit f)$ is the Hausdorff distance.

\begin{defn}\label{defn:CfBT}
We define a positive constant
\be\label{eq:C1-f}
C_{(f;B,T)}: = \min\left \{\min_{p \in N\setminus B}|df(p)|,
d_{\text H}(N\setminus B,\Crit f)\right\}
\ee
\end{defn}

By definition of $C_{(f;B,T)}$, if $q \in N \setminus B$, we have
\be\label{eq:|df(q)|}
|df(q)|, \, d(q, \Crit f) \geq C_{(f;B,T)} > 0.
\ee

\begin{lem}\label{lem:conformal} For any $f \in  C^\infty_{crit}(N;B,T)$,
$$
C_{(\delta f;B,T)} = \min_{p \in N\setminus B}|d(\delta f)(p)|
$$
whenever $\delta > 0$ is so small that
$$
\min_{p \in N\setminus B}|d(\delta f)(p)| < d_{\text{H}}(N\setminus T, B).
$$
In particular, for such $\delta > 0$,
\be\label{eq:C1conformal}
\lambda C_{(\delta f;B,T)} = C_{(\lambda \delta f;B,T)}
\ee
for any $\lambda \leq 1$.
\end{lem}
\begin{proof} First note that the distance $d_{\text{\rm H}}(N\setminus B,\Crit(\delta f))$
does not depend on $\lambda$ and that
$$
\min_{p \in N\setminus B}|\delta df(p)| = \delta \min_{p \in N\setminus B}|df(p)| \to 0
$$
as $\delta \to 0$.
Therefore the minimum in the definition
$$
C_{(\delta f;B,T)} = \min\left\{\min_{p \in N\setminus B}|d(\delta f)(p)|,
d_{\text{\rm H}}(N\setminus B, \Crit (\delta f))\right\}
$$
is realized by $\min_{p \in N\setminus B}|d(\delta f)(p)|$ for all sufficiently small
$\delta > 0$. Then the lemma follows.
\end{proof}

Now we consider the Hamiltonians $H$ adapted to the triple $(f;B,T)$ as in
the definition of Hamiltonain $C^0$-topology of $\frak{Iso}_B(o_N;T^*N)$.

\begin{lem}\label{lem:(T,f)}
Let $T \supset o_B$ in $T^*N$ and $H \in \CP C^\infty_{asc;B}$ satisfy
\be\label{eq:suppphiH1}
\phi_H^1 \equiv id
\ee
on $T$. Then we have
$$
L_f \cap o_N = \phi_{H}^1(L_f) \cap o_N
$$
whenever $H$ satisfies
\be\label{eq:distphiH}
\osc_{C^0}(\phi_H^1;o_N) < C_{(f;B,T)}.
\ee
In particular all the Hamiltonian trajectories of $H\# (f\circ \pi)$,
are constant equal to $o_p$ for some point $p \in \Crit f$ for such Hamiltonian $H$.
\end{lem}
\begin{proof} In the proof, we will denote $p \in N$ and the corresponding point
in the zero section of $T^*N$ by $o_p$ for the notational consistency.

Obviously we have $\Crit f = L_f \cap o_B \subset \phi_{H}^1(L_f) \cap o_N$
since we assume $\phi_H^1 \equiv id$ on a neighborhood, $T$, of $o_B \supset \Crit f$.

We will now prove the opposite inclusion $\phi_{H}^1(L_f) \cap o_N \subset L_f \cap o_B$.
Suppose $o_p \in \phi_{H}^1(L_f) \cap o_N$. Then we have $(\phi_{H}^1)^{-1}(o_p) \in L_f$.

Consider first the case $p \in B$.
In this case since we assume $\phi_{H}^1 = id$ on a neighborhood of $o_B$,
it in particular implies $o_p = (\phi_{H}^1)^{-1}(o_p)$
for all $i$ and hence $o_p \in o_B \cap L_f \cong \Crit f$.

Now we will show that $p$ cannot lie in $N \setminus B$.
Suppose $p \in N \setminus B$ to the contrary and write
$$
(\phi_{H}^1)^{-1}(o_p) = df(p')
$$
for some $p' \in N$. Therefore
$$
d(o_p, df(p')) = d(o_p,(\phi_{H}^1)^{-1}(o_p)) \leq \osc_{C^0}(\phi_H^1;o_N).
$$
Furthermore we also have $|df(p')| \leq d(o_p, df(p'))$ by Lemma \ref{lem:rNg}
since $\Graph df \subset D^r(T^*N)$. Therefore we have shown
\be\label{eq:upperbound}
|df(p')| \leq \osc_{C^0}(\phi_H^1;o_N) < C_{(f;B,T)}.
\ee
This in particular implies $(\phi_{H}^1)^{-1}(o_p)= df(p')$ must lie in $\Graph df|_B \subset T$ for otherwise
$|df(p')| \geq C_{(f;B,T)}$ by definition of $C_{(f;B,T)}$ which would
contradict to \eqref{eq:upperbound}.

This in turn implies $(\phi_{H}^1)^{-1}(o_p) \in T$. But $\phi_H^1$ is assumed to be the identity map
on $T$ and hence follows
$$
o_p =(\phi_{H}^1)^{-1}(o_p) = df(p').
$$
In particular $df(p') \in o_N$ and so $p' \in \Crit f$ and hence $o_{p'} = df(p')$. This
implies $p = p'$ and so $d(p, \Crit f) = 0$, i.e., $p \in \Crit f \subset B$,
a contradiction to the hypothesis $p \in N\setminus B$.
Therefore $p$ cannot lie in $N \setminus B$ and hence proves $o_p \in o_B \cap L_f \cong \Crit f$
for any $o_p \in \phi_{H}^1(L_f) \cap o_N$. This then finishes the proof of the first statement
\be\label{eq:Lfint=Ltint}
L_f \cap o_N = \phi_H^1(L_f) \cap o_N.
\ee

To prove the second statement, the first statement of the lemma implies that
all the Hamiltonian trajectories of $H \# f\circ \pi$ ending at a point
in $\phi_{H}^1(L_f) \cap o_N$ have the form
$$
z_p^{H\# f\circ \pi}(t) = \phi_{H\# f\circ \pi}^t((\phi_{H\# f\circ \pi}^1)^{-1}(o_p))
$$
for some intersection point $o_p \in \phi_{H}^1(L_f) \cap o_N = L_f \cap o_N$.
By definition, we have $z_p^{H\# f\circ \pi}(1) = o_p$.

But we also have $df(p) = 0$ and $(\phi_{H}^1)^{-1}(o_p) = o_p$ since
$$
o_p \in \phi_{H}^1(L_f) \cap o_N = L_f \cap o_N \subset o_B \cap \Crit f
$$
and $\phi_{H}^1 \equiv id$ near $p$. Therefore
$$
(\phi_{H\# f\circ \pi}^1)^{-1}(o_p) = (\phi_{f\circ \pi}^1)^{-1}(\phi_{H}^1)^{-1}(o_p) = o_p.
$$
Therefore
\beastar
z_p^{H\# f\circ \pi}(t) & = & \phi_{H\# f\circ \pi}^t((\phi_{H\# f\circ \pi}^1)^{-1}(o_p))
= \phi_{H\# f\circ \pi}^t(o_p) \\
& = &\phi_H^t(\phi_{f\circ \pi}^t(o_p)) = \phi_H^t(o_p) = o_p
\eeastar
since $df(p) = 0$ and $\phi_{H}^t(o_p) = o_p$ for all $t \in [0,1]$. The last statement
follows since we assume $\supp \phi_{H} \cap o_B = \emptyset$: By compactness of $\supp \phi_{H}$
and the closeness of $B$, $\supp \phi_{H} \cap o_B = \emptyset$
implies $\phi_{H}^t \equiv id$ for all $t \in [0,1]$
on a neighborhood $T' \supset o_B$ in $T^*N$.

This finishes the proof.
\end{proof}

\begin{rem}\label{rem:support} We would like to mention that in the above proof,
the choice of the neighborhood $T' \supset B$ is allowed to vary depending on $H$'s.
This is because our Hamiltonian $C^0$-topology requires only $\supp \phi_{H}^t \cap o_B = \emptyset$
for $t \in [0,1]$, not the existence of uniform neighborhood $T \supset o_B$ independent of $H$.
It only requires existence of such uniform neighborhood for the time-one map $\phi_H^1$.
\end{rem}

\begin{rem} In fact all the discussion in this subsection can be generalized by
replacing the differential $df$ by any closed one form $\alpha$ and $\Crit f$ by
the zero set of $\alpha$. But we restrict to the exact case since
the discussion in the next subsection seems to require the exactness of the form.
\end{rem}

\subsection{Lagrangian capacity versus Hamiltonian $C^0$-fluctuation}
\label{subsec:capacity-fluctuation}

In fact, Theorem \ref{thm:rhoC0conti} is an immediate consequence of the following
comparison result between the Lagrangian capacity
$\gamma^{lag}_B(L) = \rho^{lag}(H;1) - \rho^{lag}(H;[pt]^\#$ and the
Hamiltonian $C^0$-fluctuation $\osc_{C^0}(\phi_H^1;o_N)$ for $L = \phi_H^1(o_N)$
for $H \in \CP^\infty_{asc;B}$, which itself has some independent interest in its own right.

\begin{thm}\label{thm:capacity} Let $B \subset N$ be a closed ball and
$(T,f) \in \CT_{(B;r)}$. Consider the set of Hamiltonians $H$
satisfying $\supp \phi_H \cap o_B =\emptyset$ and assume
$$
\osc_{C^0}(\phi_H^1;o_N) < C_{(f;B,T)}.
$$
Then we have
\be\label{eq:ratio}
\frac{\gamma^{lag}_B(L)}{\osc_{C^0}(\phi_H^1;o_N)} \leq
\frac{2\, \osc f}{C_{(f;B,T)}}
\ee
for $L = \phi_H^1(o_N)$.
\end{thm}
We would like to mention that the right hand side of \eqref{eq:ratio} does not depend on
the scale change of $f$ to $\delta\, t$ for $\delta > 0$.

The following question seems to be an interesting question to ask
in regard to the precise estimate of the upper bound in this theorem and Question \ref{ques:ham-continuity}.

\begin{ques} For given $H$ satisfying the condition
in Theorem \ref{thm:capacity}, what is an optimal estimate of the
constant $\frac{2\, \osc f}{C_{(f;B,T)}}$ in terms of
$B$, $T$ and $H$? For example, can we obtain an upper bound independent of $B$ or $T$?
\end{ques}

The rest of the section is occupied by the proof of Theorem \ref{thm:capacity}.
The following proposition is a crucial ingredient of the proof, which is a variation of
Proposition 2.6 \cite{ostrov}, Proposition 3.3 \cite{entov-pol:morphism},
Proposition 3.1 \cite{usher:sharp} and Proposition 2.3 \cite{seyfad}.

\begin{prop}\label{prop:rhoLHoscf}
Let $H \in \CP C^\infty_{asc}$ in $T^*N$ such that
\be\label{eq:suppB2}
\supp \phi_H \cap o_B = \emptyset.
\ee
Take any $f \in C^\infty_{crit}(N;B)$ such that \eqref{eq:distphiH} holds. Then
\be\label{eq:rhoH1C1}
\rho^{lag}(H;1) - \rho^{lag}(H;[pt]^\#) \leq 2 \, \osc f.
\ee
\end{prop}
\begin{proof} Denote
$
L_f: = \Graph df, \quad L_t = \phi_H^t(L_f) = \phi_H^t(\Graph df).
$
Note that the condition \eqref{eq:suppB2} implies
\be\label{eq:HconstonB}
H_t|_B \equiv c_B(t)
\ee
for a function $c_B=c_B(t)$ depending only on $t$ but not on $x \in B$.

The following lemma is the analogue of Lemma 5.1 \cite{ostrov}.

\begin{lem}\label{lem:hLs=hLf}
\be\label{eq:Lt=Lf}
\rho^{lag}(H\# f;1) - \rho^{lag}(H\# f;[pt]^\#) \leq \osc f.
\ee
\end{lem}
\begin{proof}
By the spectrality of $\rho^{lag}(\cdot, 1)$ in general, we have
\beastar
\rho^{lag}(H \# f\circ \pi;1) & = & \CA^{cl}_{(H \# f\circ \pi)}\left(z^{H \# f\circ \pi}_{p_-}\right), \\
\rho^{lag}(H\# f\circ \pi;[pt]^\#) & = &\CA^{cl}_{(H \# f\circ \pi)}\left(z^{H \# f\circ \pi}_{p_+}\right)
\eeastar
for some $p_\pm \in L_f \cap o_N$. Using the second statement of Lemma \ref{lem:(T,f)},
we compute
\beastar
&{}& \CA^{cl}_{(H \# f\circ \pi)}\left(z^{H \# f\circ \pi}_{p_+}\right)
- \CA^{cl}_{(H \# f\circ \pi)}\left(z^{H \# f\circ \pi}_{p_-}\right)\\
& = & - \int_0^1 (H \# f\circ \pi)(t,p_+) \, dt +
 \int_0^1 (H \# f\circ \pi)(t,p_-) \, dt\\
& = & - \int_0^1 c_B(t)\, dt- f(p_+)
+ \int_0^1 c_B(t)\, dt  + f(p_-)  \\
& = & - f(p_+) + f(p_-) \leq  \max f - \min f = \osc f.
\eeastar
Here for the equality in the line next to the last, we use the identity
$$
(H \# f\circ \pi)(t,p_\pm) = H(t, p_\pm) + f(\phi_H^{t}(p_\pm))
= c_B(t) + f(p_\pm).
$$
This finishes the proof.
\end{proof}

On the other hand,  we have
$$
\phi_H^1(L_f)= \phi_H^1(\phi_{f\circ \pi}^1(o_N)) = \phi_{H\# f\circ\pi}^1(o_N)
$$
and so by the triangle inequality, Proposition \ref{prop:oh-triangle},
\beastar
\rho^{lag}(H \# (f\circ \pi);1) & \geq & \rho^{lag}(H;1) - \rho^{lag}(-f\circ \pi ;1) \\
\rho^{lag}(H \# (f\circ \pi);[pt]^\#) & \leq & \rho^{lag}(H;[pt]^\#) + \rho^{lag}(f\circ \pi;1).
\eeastar
(One can also use Proposition \ref{prop:Lag-triangle} using the concatenation
$H * (f \circ \pi)$ instead. Here
$f\circ \pi$ is not boundary flat, which is required in
Proposition \ref{prop:Lag-triangle}, but one can always reparameterize
the flow $t \mapsto \phi_{f\circ \pi}^t$ by multiplying
$\chi'(t)$ to $f \circ \pi$ so that the perturbation is as small as we want
in $L^{(1,\infty)}$-topology which in turn perturbs $\rho$ slightly.
See Lemma 5.2 \cite{oh:ajm1},
Remark 2.5 \cite{MVZ} for the precise statement on this
approximation procedure.
This enables us to apply the triangle inequality in Proposition \ref{prop:Lag-triangle}
in the current context.)

Therefore subtracting the second inequality from the first and using the identity
$$
\rho^{lag}(-f\circ \pi ;1) = \max f, \quad \rho^{lag}(f\circ \pi;1) = - \min f
$$
(see \cite{oh:cag} for its proof),
we obtain
\beastar
&{}& \rho^{lag}(H \# (f\circ \pi);1) - \rho^{lag}(H \# (f\circ \pi);[pt]^\#) \\
& \geq & \rho^{lag}(H;1) - \rho^{lag}(H;[pt]^\#) -(\max f - \min f)
\eeastar
which in turn gives rise to
\beastar
\rho^{lag}(H;1) - \rho^{lag}(H;[pt]^\#) & \leq & \rho^{lag}(H \# (f\circ \pi);1)
- \rho^{lag}(H \# (f\circ \pi);[pt]^\#)\\
&{}& +(\max f - \min f)\\
&\leq&  2\, \osc f.
\eeastar
We have finished the proof of the proposition.
\end{proof}

We now go back to the proof of Theorem \ref{thm:capacity}.

Let $H \in \CP C^\infty_{asc;B}$ and $T \supset o_B$ such that
$\phi_H^1 \equiv id$ on $T$ and assume \eqref{eq:distphiH}.

If $\osc_{C^0}(\phi_H^1;o_N) = 0$,
we have $\phi_H^1(o_N) = o_N$
and so $\rho^{lag}(H;1) - \rho^{lag}(H;[pt]^\#) = 0$ for which \eqref{eq:rhoH1C1}
obviously holds. Therefore we assume $\osc_{C^0}(\phi_H^1;o_N)\neq 0$.

Recall from Lemma \ref{lem:(T,f)} that the choice of $f$ depends only on the ball $B$ and
the neighborhood $T \supset o_B$ in $T^*N$.
Then we choose $\lambda > 0$ such that
$$
\osc_{C^0}(\phi_{H}^1;o_N) = \lambda C_{(f;B,T)}
$$
i.e.,
$$
\lambda = \frac{\osc_{C^0}(\phi_{H}^1;o_N)}
{C_{(f;B,T)}}.
$$
Obviously we have
$$
\osc_{C^0}(\phi_{H}^1;o_N) < (\lambda + \e) C_{(f;B,T)}
$$
for all $\e > 0$. We note that both $ d_{\text{\rm H}}(N\setminus B, \Crit ( \delta f))$ and the ratio
$
\frac{2 \osc f}{C_{(f;B,T)}}
$
do not depend on the choice of $\delta> 0$.

Therefore we can replace $f$ by $\delta f$ for a sufficiently small $\delta > 0$, if necessary, so
that
\be\label{eq:min<d}
\min_{p \in N \setminus B}|d(\lambda (\delta f))(p)| < d_{\text{\rm H}}(N\setminus B, \Crit ( \delta f)
\ee
which in turn implies
$$
\lambda C_{(\delta f;B,T)} = C_{(\lambda \delta f;B,T)}
$$
by Lemma \ref{lem:conformal}. From now on, we assume
\be\label{eq:min<d}
\min_{p \in N \setminus B}|d(\lambda f)(p)| < d_{\text{\rm H}}(N\setminus B, \Crit f)
\ee
without loss of any generality.

Lemma \ref{lem:conformal} also implies
$$
(\lambda+\e) C_{(f;B,T)} = C_{((\lambda+\e) f;B,T)}
$$
for all small $\e > 0$ such that
$$
\min_{p \in N \setminus B}|(\lambda+\e) df(p)| < d(N\setminus B,\Crit f).
$$
For example, we can choose any $\e > 0$ so that
\be\label{eq:choicee}
0< \e < \frac{d(N\setminus B, \Crit f)}{\min_{p \in N \setminus B}|df(p)|}.
\ee

Since \eqref{eq:rhoH1C1} holds for any pair $H,\, f$ that satisfy
\eqref{eq:suppB2} and \eqref{eq:distphiH}, applying it to the pair $(H,(\lambda+\e)f)$
for $T\supset B$ chosen above independently of $i$'s, we derive
\beastar
\rho^{lag}(H;1) - \rho^{lag}(H;[pt]^\#) & \leq & 2 \osc ((\lambda+\e) f) = 2
(\lambda+\e) \,\osc f\\
& = & 2\left(\frac{\osc_{C^0}(\phi_{H}^1;o_N)}
{C_{(f;B,T)}} + \e \right)\osc f.
\eeastar
Since this holds for all $\e > 0$ satisfying \eqref{eq:choicee}, it follows
\be\label{eq:rhoversusC0}
0 \leq \rho^{lag}(H;1) - \rho^{lag}(H;[pt]^\#) \leq
2\left(\frac{\osc f}{C_{(f;B,T)}}\right) \osc_{C^0}(\phi_{H}^1;o_N)
\ee
letting $\e \to 0$. This finishes the proof of Theorem \ref{thm:capacity}.
\qed

\end{document}